\documentclass[11pt, english,english]{article}
\usepackage{amsmath,amssymb,amsthm} 

\usepackage{mathrsfs,bm,enumerate}
\usepackage{graphicx,color,tikz}
\usepackage{caption,subcaption}
\usepackage{babel}
\usepackage[T1]{fontenc} 
\usepackage{geometry}
\usepackage[latin9]{inputenc}
\usepackage{amsthm}
\usepackage{lipsum}
\newcommand\blfootnote[1]{%
  \begingroup
  \renewcommand\thefootnote{}\footnote{#1}%
  \addtocounter{footnote}{-1}%
  \endgroup}
\newcommand{\inR}{\in \mathbb{R}}
\newcommand{\inZ}{\in \mathbb{Z}}

\newcommand{\R}{ \mathbb{R}}
\newcommand{\Z}{ \mathbb{Z}}
\newcommand{\T}{\mathbb{T}}
\newcommand{\N}{ \mathbb{N}}

\newcommand{\Lop}{{\rm L}}
\newcommand{\Dop}{{\rm D}}

\newcommand{\dint}{{\rm d}}
\newcommand{\Fourier}{ \mathcal{F}} 

\newcommand{\bx}{{\boldsymbol x}}
\newcommand{\bw}{{\boldsymbol \omega}}
\newcommand{\bk}{{\boldsymbol k}}
\newcommand{\bl}{{\boldsymbol l}}

\renewcommand{\[}{\begin{equation}}
\renewcommand{\]}[1]{\label{eq:#1}\end{equation}}

\providecommand{\abs}[1]{\left\lvert#1\right\rvert}

\providecommand{\set}[1]{ \left\{ #1  \right\}  }
\providecommand{\setb}[2]{ \left\{ #1 \ \middle| \  #2 \right\}  }

\providecommand{\parenth}[1]{\left( #1 \right) }

\def\V#1{{\boldsymbol{#1}}}         
\def\Spc#1{{\mathcal{#1}}}  
\def\Op#1{{\mathrm{#1}}}  
\def\ee{\mathrm{e}} 
\def\jj{\mathrm{i}}

\newtheorem{definition}{Definition}
\newtheorem{proposition}{Proposition}
\newtheorem{corollary}{Corollary}

\newtheorem{lemma}{Lemma}
\newtheorem{theorem}{Theorem}

\title{Beyond Wiener's Lemma:\\ Nuclear Convolution Algebras\\ and the Inversion of Digital Filters\thanks{The research leading to these results has received funding from the European Research Council under the European Union's Seventh Framework Programme (FP7/2007-2013) / ERC grant agreement $\text{n}^\circ$ 267439. 
}}

\author{
Julien Fageot  \thanks{Biomedical Imaging Group, \'Ecole polytechnique f\'ed\'erale de Lausanne (EPFL),
Station 17, CH-1015, Lausanne, Switzerland ({\tt julien.fageot@epfl.ch, michael.unser@epfl.ch, jpward@ncat.edu}). }
\and
Michael Unser \footnotemark[2]
\and
John Paul Ward
\thanks{
Department of Mathematics, North Carolina A\&T State University, Greensboro, NC 27411, USA
}
 }

\begin{document}

\maketitle

\begin{abstract} A convolution algebra is a topological vector space $\Spc X$ that is closed under the convolution operation.  It is said to be inverse-closed if each element of $\Spc X$ whose spectrum is bounded away from zero has a convolution inverse that is also part of the algebra.
The theory of discrete Banach convolution algebras is well established with a complete characterization of the weighted $\ell_1$ algebras that are inverse-closed---these are henceforth referred to as the Gelfand-Raikov-Shilov (GRS) spaces. 

Our starting point here is the observation that the space $\Spc S(\Z^d)$ of rapidly decreasing sequences, {which is not Banach but nuclear}, is an inverse-closed convolution algebra. 
This property propagates to the more constrained space of exponentially decreasing sequences $\Spc E(\Z^d)$ that we prove to be nuclear as well.
Using a recent extended version of the GRS condition, we then show that $\Spc E(\Z^d)$ is actually the smallest inverse-closed convolution algebra. This allows us to describe the hierarchy of the inverse-closed convolution algebras from the smallest, $\Spc E(\Z^d)$, to the largest, $\ell_{1}(\Z^d)$. 
In addition, we prove that, in contrast to $\Spc S(\Z^d)$,  all members of $\Spc E(\Z^d)$ admit well-defined convolution inverses in $\Spc S'(\Z^d)$ with the ``unstable'' scenario (when some frequencies are vanishing) giving rise to inverse filters with slowly-increasing impulse responses.
{Finally, we use those results to reveal the decay and reproduction properties of an extended family of cardinal spline interpolants.} 


\end{abstract}

\section{Introduction}

The task of inverting a digital filter is referred to as ``deconvolution'' in signal and image processing. It allows for the deblurring of digital images. Being able to invert discrete convolution operators is also crucial in approximation and sampling theory in order to specify interpolators and dual basis functions \cite{Aldroubi2001b,Aldoubi2001,Unser2000}.
It is well known that a $\ell_2$-stable inversion is possible if and only if the modulus of the discrete Fourier transform $\hat h=\Fourier_{\rm d}\{h\}$ of the filter is bounded from above and below \cite{Grafakos2008}---a condition that we shall refer to as ``$\ell_2$-invertibility''. 
The inverse filter $g$ is then simply given by $g=\Fourier_{\rm d}^{-1}\{1/\hat h\}$. The theoretical question of interest then is to characterize the stability and decay properties of the inverse filter $g$. 

The foundational result of this classical line of research in harmonic analysis is Wiener's lemma which states that, if $h\in \ell_1(\Z^d)$ and is invertible, then the convolution inverse $g$ is included in  $\ell_1(\Z^d)$ as well \cite{Groechenig2010,Krishtal2011,Newman1975,Wiener1932}. The other  ingredient is Young's inequality (the simple case with $p=1$) for discrete convolution, stating that
\begin{align}
\label{Eq:Young}
\|h \ast a\|_{\ell_1(\Z^d)}\le \|h\|_{\ell_1(\Z^d)} \|a\|_{\ell_1(\Z^d)}\end{align}for any $a,h \in \ell_1(\Z^d)$, 
where the discrete multidimensional convolution between $h$ and $a$ is defined by
\begin{align}
\label{Eq:conv}
(h \ast a)[\cdot]=\sum_{\bk \inZ^d} h[\cdot-\bk]a[\bk].
\end{align}
The norm inequality \eqref{Eq:Young} has two important consequences. The first is that the convolution operator $a \mapsto h \ast a$ is bounded on $\ell_1(\Z^d)$ if and only if\footnote{The reverse implication is obtained by taking the input to be the discrete impulse (neutral element of the convolution).} $h \in \ell_1(\Z^d)$. The second is that
$\ell_1(\Z^d)$ is a Banach convolution algebra, meaning that it is closed with respect to convolution; \emph{i.e.},
if $a, h \in \ell_1(\Z^d)$, then $h \ast a \in \ell_1(\Z^d)$.
A number of refinements of this characterization can be obtained via the specification of more constraining convolution algebras \cite{Feichtinger1979,Groechenig2010}.

The inverse-closedness of such convolution algebras can then be ensured with the help of extended versions of Wiener's lemma for weighed 
$\ell_1$-norms, the most general form being attributed to Gelfand, Raikov, and Shilov \cite{grochenig2007weight,Groechenig2010,Groechenig2004,Krishtal2011,Sun2007}. In particular, these results ensure that the convolution inverse of a filter that is algebraically decreasing retains the property, with the same order of decay.

As suggested by the title, our objective in this paper is to move beyond Wiener's lemma and the decay limit imposed by the Gelfand-Raikov-Shilov (GRS) criterion. Since the latter is an ``if and only if''  characterization, this is only possible outside the traditional realm of weighted Banach spaces.
{
Our proposal therefore is to switch to projective and inductive limits of Banach spaces, which provide a rigorous framework for extending classical results from functional analysis to non-Banach spaces. As it turns out, the relevant spaces have a nuclear structure.
}
To quote A. Pietsch  in \cite[pp. V]{Pietsch1972nuclear}: 

 \noindent \textit{``With a few exceptions the locally convex spaces encountered in analysis can be divided into two classes. First, there are the normed spaces, which belong to classical functional analysis, and whose theory can be considered essentially closed. The second class consists of the so-called nuclear locally convex spaces, which were introduced in 1951 by A. Grothendieck."}

\noindent While this requires the use of a more advanced formalism \cite{Grothendieck1955, Schaefer1999, Treves2006}, the payoff in our case is that we end up with stronger results on the decay of convolution inverses.

 Before presenting our contributions, we should mention important previous works about the inverse-closedness of non-Banach convolution algebras. Several authors have studied the space of exponentially decreasing sequences and have shown that it is stable under convolution and inversion \cite{Almira2006inverse,Demko1984decay,Jaffard1990proprietes,Reed1980methods}. As we shall see, this space plays a fundamental role within the complete family of inverse-closed convolution algebras. More recently, Fern\'andez, Galbis and Toft extended the GRS condition to the case of countable inductive limits of Banach convolution algebras \cite{Fernandez2014spectral}; on this subject, see also \cite{Baskakov1990wiener,Baskakov1997estimates,fernandez2015characterizations}. We should emphasize that these works have been developed not only for convolution operators, but for the more general framework of infinite-dimensional matrices whose elements are dominated by convolution kernels. 

As already announced, our focus here is on convolution. Our objectives are essentially two-fold.
First, we are aiming at a unified and self-contained treatment of inverse-closed convolution algebras for sequence spaces beyond the classical setting of Banach spaces, with the widest possible range of applicability. Working with convolution allows us to use extensively the discrete Fourier transform.  An interesting aspect is that some of the proofs become deceptively simple once the problem has been correctly specified. Second, we are introducing a classification and hierarchy of inverse-closed convolution algebras which helps us  delineate the boundaries of the framework.  The main points that are developed thereafter are as follows.

\begin{itemize}
\item We give simple proofs that the space of rapidly decreasing sequences $\Spc S(\Z^d)$ and the space of exponentially decreasing sequences $\Spc E(\Z^d)$ are inverse-closed nuclear convolution algebras. Our results on the nuclearity of $\Spc E(\Z^d)$ are new, to the best of our knowledge. 

\item We characterize the inverse-closedness of countable intersections and countable unions of Banach convolution algebras. For the union, we use the extended GRS condition, and show that it is equivalent to the inverse-closedness of the corresponding convolution algebra (the sufficiency was proved in \cite{Fernandez2014spectral}).

\item We prove that the space $\Spc E(\Z^d)$ is the projective limit of the GRS spaces (Theorem \ref{theo:icing}). It means that the space of exponentially decreasing sequences is the intersection of the complete family of inverse-closed Banach convolution algebras. 
We then deduce that $\Spc E(\Z^d)$ is the smallest inverse-closed convolution algebra among sequence spaces (Theorem \ref{theo:Esmallest}). These are possibly our most important results.

\item We prove that all members of $\Spc E(\Z^d)$, including those that have frequency nulls, admit a convolution inverse in the space of slowly increasing sequences $\Spc S'(\Z^d)$ (Theorem \ref{Theo:geninverse}). This is an improvement of an earlier result by De Boor, H\"ollig and Riemenschneider \cite{DeBoor1989}, who established the unconditional invertibility in $\Spc S'(\Z^d)$ of the more restricted family of compactly-supported filters. 

\item {
We demonstrate the relevance of our results for the problem of cardinal spline interpolation \cite{Meijering2002,Schoenberg1973}; that is, the determination of an interpolating spline taking predetermined values at the integers. 
Specifically, given some integer shift-invariant space
$V_{\varphi}={\rm span}\{\varphi(\cdot-\V k)\}_{\V k \in \Z^d}$, we relate the decay of the unique interpolant $\varphi_{\rm int} \in V_{\varphi}$ to the functional properties of the generator $\varphi$. In particular, we identify general conditions on $\varphi$ 
(resp., $\hat \varphi$) that ensure that $\varphi_{\rm int}$ decays exponentially fast.
These cover the classical case of polynomial spline interpolation where $\varphi$ is a compactly-supported B-spline \cite{Chui1987,Schoenberg1973}, as well as 
the more challenging scenario where $\varphi$ is the slowly increasing Green's function of some (elliptic) differential operator $\Lop$, extending the results of Madych and Nelson for the polyharmonic splines where $\Lop=(-\Delta)^n$ is the $n$-fold Laplacian \cite{madych1990}. 

}
\end{itemize}

The paper is organized as follows. In Section \ref{sec:notations}, we introduce some notations and definitions. In Section \ref{section:review}, we review the results on inverse-closed Banach convolution algebras. In Section \ref{section:EandS}, we study the spaces $\Spc S(\Z^d)$ and $\Spc E (\Z^d)$, which are two prominent examples of   non-Banach\footnote{A nuclear space cannot be normed, unless the number of dimensions is finite. The property of nuclearity comes hand-in-hand with the specification of the underlying topology and a characterization of the dual as an inductive or projective limit of Banach spaces.
} 
 inverse-closed convolution algebras with a well-defined topology. Our main results are presented  in Section \ref{section:main}, where we introduce the classification of inverse-closed convolution algebras and show that $\Spc E(\Z^d)$ is actually the smallest one. 
{
In Section \ref{sec:singular}, we consider the inversion of sequences of $\Spc E(\Z^d)$ with frequency nulls.
Finally, we conclude in Section \ref{sec:application} with the application of our results to the problem of interpolation on a uniform grid, which requires the inversion of a discrete convolution operator.
}
\section{Notations and Definitions} \label{sec:notations}

Following the standard convention in signal processing, we use square brackets to index sequences (\emph{e.g.}, discrete signals) and round ones to index functions (\emph{e.g.}, continuous-domain signals).
In this way, we can denote the sampled value of some function $f$ at some (multi-)integer location $\bx=\bk$ as $\left.f[\bk]=f(\bx)\right|_{\bx =\bk}$, while $f[\cdot]=(f[\bk])_{\bk \in \Z^d}$ represents the sampled version of the function $f(\cdot)$ on the multi-integer grid $\Z^d$.
For $\bm{x} =(x_1,\ldots,x_d) \in \R^d$, we denote $\lVert \bm{x} \rVert = (x_1^2+\cdots +x_d^2)^{1/2}$ the Euclidian norm and $\lvert \bm{x} \rvert = \abs{x_1}+\cdots + \abs{x_d}$ the $\ell_1$-norm.
	
	\subsection{Sequences Spaces}

\begin{definition}
\label{Def:signals}
A discrete multidimensional signal $a[\cdot]$ is 
\begin{itemize}
\item  \emph{rapidly decreasing} if  $\sup_{\bk \in \Z^d} (1 +\|\bk\|)^n|a[\bk]| < \infty$ for all $n \in \N$; that is, if  $a[\cdot]$  has a faster-than-algebraic decay at infinity;
\item \emph{exponentially decreasing} if there exists a constant $C$ and a rate $r>0$ such that $|a[\bk]| \leq C \ee^{-r\abs{\bm{k}}}$ for all $\bk \inZ^d$;
\item \emph{slowly increasing} if there exists a constant $C$ and an integer $n\in \N$ such that $|a[\bk]| \leq C (1+\|\bk\|)^n$ for all $\bk \inZ^d$; that is, if  $a[\cdot]$  does not grow faster than all polynomials.
\end{itemize}
The corresponding vector spaces are denoted $\Spc S(\Z^d)$, $\Spc E(\Z^d)$, and $\Spc S'(\Z^d)$ respectively.
\end{definition}

The above spaces can be endowed with natural topologies that make them topological vector spaces.
We detail their topological structures motivated by the fact that they are not Banach spaces. They are actually nuclear spaces \cite{Treves2006}, as will be made explicit in the sequel.
The space that has the simplest structure is
\begin{align} 
\label{eq:S}
\Spc S(\Z^d)&=  \bigcap_{n \in \mathbb{N}} \Big\{ a[\cdot]: \sup_{\bk \inZ^d} (1+\|\bk\|)^n |a[\bk]| < \infty \Big\},
\end{align}
which is a countable projective limit of Banach spaces, also called a Fr\'echet space. In particular, a sequence $(a_m)_{m \in \mathbb{N}}$ of elements of $\Spc S(\Z^d)$ converges to $a \in \Spc S(\Z^d)$ if, for every $n\in \mathbb{N}$, 
\begin{equation*}
	 \sup_{\bk \inZ^d} (1+\|\bk\|)^n |(a_m - a)[\bk]| \underset{m\rightarrow \infty}{\longrightarrow} 0.
\end{equation*}
 
By contrast, the spaces $\Spc E(\Z^d)$ and $\Spc S'(\Z^d)$ are countable inductive limits of Banach spaces, since we can write
\begin{align}
\Spc E(\Z^d)  &= \bigcup_{n\in \mathbb{N}} \Big\{ a[\cdot]: \sup_{\bk \inZ^d} \mathrm{e}^{-\abs{\bm{k}}/(n+1)   } \abs{a[\bm{k}]} < \infty \Big\}, \label{eq:E(Z)} \\
\Spc S'(\Z^d) &= \bigcup_{n\in \mathbb{N}} \Big\{ a[\cdot]: \sup_{\bk \inZ^d} (1+\|\bk\|)^{-n} \abs{a[\bm{k}]} < \infty \Big\}.  \label{eq:S'(Z)}
\end{align}
A sequence $(a_m)_{m \in \mathbb{N}}$ of elements of $\Spc S'(\Z^d)$ converges to $a\in \Spc S'(\Z^d)$ if  there exists $n_0 \in \mathbb{N}$ such that $\sup_{\bk \inZ^d} (1+\|\bk\|)^{-n_0} \abs{a_m[\bm{k}]} < \infty$ for all $m$, $\sup_{\bk \inZ^d} (1+\|\bk\|)^{-n_0} \abs{a[\bm{k}]} < \infty$, and
\begin{equation*}
	\sup_{\bk \inZ^d} (1+\|\bk\|)^{-n_0} \abs{(a_m-a)[\bm{k}]} \underset{m\rightarrow \infty}{\longrightarrow} 0.
\end{equation*}
The convergence in $\Spc E(\Z^d)$ obeys the same principle.

The space $\Spc S'(\Z^d)$ is moreover the topological dual of $\Spc S(\Z^d)$; that is, the space of continuous linear functionals from $\Spc S(\Z^d)$ to $\R$. Actually, we shall see that $\Spc E(\Z^d)$ is, like $\Spc S'(\Z^d)$, the topological dual of a Fr\'echet space. Thereafter, duals of topological vector spaces are endowed with the strong topology\footnote{Note that the strong topology on $\Spc S'(\Z^d)$ coincides with the inductive topology of \eqref{eq:S'(Z)}.}.

We say that a topological  vector space $E$ embeds into a topological vector space $F$, denoted as $E \subseteq F$, if $E$ is included in $F$ (set inclusion) and if the identity operator is continuous from $E$ to $F$. We have the following embedding relations
$$
\Spc E(\Z^d)\subseteq\Spc S(\Z^d)  \subseteq \ell_p(\Z^d) \subseteq     \Spc S'(\Z^d)
$$
with the property that any of the classical Banach spaces $\ell_p(\Z^d)$ with $1\leq p \leq \infty$ is sandwiched in-between.

	\subsection{Convolution Algebras and Inverse-Closedness}

The discrete convolution of two sequences $a$ and $b\in \Spc S'(\Z^d)$ is well-defined if $\sum_{\bm{l} \in \Z^d} \abs{a[\bm{l}]b[\bm{k}-\bm{l}]} <\infty$ for every $\bm{k} \in \Z^d$, in which case we set
		$$a*b = \sum_{\bm{l}\in \Z^d} a[\cdot - \bm{l}] b [\bm{l}].$$
The Kronecker delta $\delta[\cdot]$, that is $1$ at $\V 0$ and $0$ otherwise, is the neutral element for the convolution. 
 
\begin{definition}
	A (discrete) \emph{convolution algebra} is a sequence space $\Spc X \subseteq  \Spc S'(\Z^d)$ such that
	\begin{itemize}
		\item the Kronecker delta $\delta[\cdot]$ is in $\Spc X$;
		\item for all $a,b\in \Spc X$, the convolution $a*b$ is well-defined and is in $\Spc X$;
		\item the convolution defines a bilinear and continuous operator from $\Spc X \times \Spc X$ to $\Spc X$.
	\end{itemize}
	Moreover, if $\Spc X$ is a Banach space (resp., a nuclear space), we say that $\Spc X$ is a \emph{Banach convolution algebra} (resp., a \emph{nuclear convolution algebra}). 
\end{definition}

For a topological vector space $\Spc X \subseteq \Spc S'(\Z^d)$, we define $\Spc A (\Spc X)$ as the space of continuous linear shift-invariant operators from $\Spc X$ to itself. A shift-invariant operator $\Op T$ in $\Spc A(\Spc X)$ being a convolution, it is identified with its impulse response $h \in \Spc S'(\Z^d)$ such that $\Op T \{a\} = h*a$ for every $a\in \Spc X$. We denote the latter operator $\Op T_h$.  Then, $\Spc A(\Spc X)$ is isomorphically equivalent to a sequence space.
The case of a convolution algebra is characterized by the relation $\Spc A (\Spc X) = \Spc X$. Indeed, the convolution being continuous from $\Spc X \times \Spc X$ to $\Spc X$, the operator $\Op T_h$ is continuous from $\Spc X$ to itself for every $h \in \Spc X$, so $\Spc X \subset \Spc A (\Spc X)$. Reciprocally, if $\Op T_h \inÊ\Spc A (\Spc X)$, then, since $\delta[\cdot] \in \Spc X$, $h = \Op T_h \{\delta \} = h*\delta = h \in \Spc X$. For instance, if $\Spc X=\ell_1(\Z^d)$, we can invoke Young's inequality to deduce that   $\ell_1(\Z^d)$ is a convolution algebra; \emph{i.e.}, $\Spc A(\ell_1(\Z^d))=\ell_1(\Z^d)$.

\begin{definition} 
\label{Def:invertibity} A filter $h$  is said to be \emph{$\ell_2$-invertible} if $h*a$ is well-defined for every $a\in \ell_2(\Z^d)$ and there exist two constants $0<A,B<\infty$ such that $$A \|a\|_{\ell_2(\Z^d)}\le\|a \ast h\|_{\ell_2(\Z^d)} \le B \|a\|_{\ell_2(\Z^d)}, \quad \forall a\in \ell_2(\Z^d).$$
\end{definition}

The discrete-domain Fourier transform (or frequency response) of $h\in \Spc S'(\Z^d)$ is given by
$$
\hat h(\bw)=\Fourier_{\rm d}\{h\}(\bw)=\sum_{\bk \inZ^d} h[\bk] \ee^{-\mathrm{i} \langle \bw, \bk\rangle},$$
where the convergence holds in $\Spc S'(\T^d)$, the space of generalized functions on the $d$-dimensional torus $\mathbb{T}^d=[-\pi,\pi]^d$.
If $h \in \ell_1(\Z^d)$, then $\hat h(\bw)$ is defined pointwise, and is actually bounded and continuous over $\T^d$.
In that case, the invertibility of $h$ is equivalent to the simple condition $\hat h(\bw)\ne0$ for all $\bw \in \mathbb{T}^d$, and the inverse is $g = \Fourier_{\rm d}^{-1} \{1/\hat{h} \}$, where $\Fourier_{\rm d}^{-1}$ is the discrete inverse Fourier transform.

\begin{definition} 
\label{Def:incclosed}A  convolution algebra $\Spc X \subseteq \Spc S'(\Z^d)$  is said to be \emph{inverse-closed} if, for any $\ell_2$-invertible filter $h \in \Spc X$, there exists $g\in\Spc X$ such that $h*g$ and $g*h$ are well-defined and $$h\ast g=g\ast h=\delta[\cdot].$$
\end{definition}

\section{Review of Classical Results on Banach Convolution Algebras} \label{section:review}

We briefly introduce the subject of Banach convolution algebras for sequence spaces. 
For more  details and more general approaches, we refer the reader to the works of Feichtinger \cite{Feichtinger1979} and Gr\"ochenig \cite{Groechenig2010}.

	\subsection{Weighted Banach Convolution Algebras}

We say that  $w : \Z^d \rightarrow \R$  is a weighting sequence if it is  positive and symmetric; that is, if $w[\bm{k}] = w[- \bm{k}] > 0$ for every $\bm{k} \in \Z^d$.
It is convenient mathematically to describe/control the decay properties of signals via their inclusion in some appropriate weighted $\ell_p$-space
\begin{align}
\ell_{p,w}(\Z^d)=\left\{ a \in \Spc S'(\Z^d):\|a\|_{\ell_{p,w}(\Z^d)}= \|w[\cdot] a[\cdot]\|_{\ell_p(\Z^d)}<\infty \right\}
\label{eq:lpw}
\end{align}
where $w[\cdot]$ is a fixed weighting sequence.
We say that the weighting sequence $w[\cdot]$  is submultiplicative if $w[\bm{k} + \bm{l}] \leq w[\bm{k}] w[\bm{l}]$ for every $\bm{k}, \bm{l} \in \Z^d$. 
Note that this condition implies that $w[\bm{0}]\ge 1$. 
Another slightly less obvious consequence is that the growth of the sequence $w[\cdot]$ is necessarily bounded by an exponential:
\begin{align}
\label{eq;multexpo}
1 \le w[\bm{k}] \le \ee^{r \lvert \bm{k} \rvert}
\end{align}
with $r= \log ( \sup_{|\bm{k}|\le1} w[\bm{k}] ) \ge0$.

\begin{theorem} 
\label{Prop:algebra1}
The space $\ell_{1,w}(\Z^d)$ is a Banach convolution algebra 
if and only if the weighting sequence $w$ is submultiplicative.
\end{theorem}
We cannot resist to reproduce the short proof of this powerful result since we find it quite enlightening (see also \cite{Feichtinger1979, Groechenig2010}).
\begin{proof}
For the direct part, we note that the submultiplicativity of $w$ is equivalent to $0<w[\bk]=w[\bk-\V l+\V l]\le w[\bk-\V l]w[\V l]$, which results in the pointwise estimate
\begin{eqnarray*}
w[\bk]\,\big|(h \ast a)[\bk]\big|&=&\left| \sum_{\V l \in \Z^d} h[\V l] a[\bk-\V l]\right|w[\bk] \\
&\le& \sum_{\V l \in \Z^d} \left(\big|h[\V l]\big|\,w[\V l]\right)\, \left(\big|a[\bk-\V l]\big|\,w[\bk-\V l]\right)
\end{eqnarray*}
Next, we evaluate the $\ell_1$ norm of both sides of the above inequality as
\begin{align*}
\|h \ast a\|_{\ell_{1,w}(\Z^d)}&\le \sum_{\V l \in \Z^d} \sum_{\V k \in \Z^d} \left(w[\V l]\big|h[\V l]\big|\right) \left(w[\V k - \V l]\big|a[\V k-\V l]\big|\right)  \\
&\ \ = \sum_{\V l \in \Z^d} \sum_{\V n \in \Z^d}  \left(w[\V l]\big|h[\V l]\big|\right)\left( w[\V n] \big| a[\V n] \big| \right)  \tag{{change of variable $\V n=\V k-\V l$}} \nonumber\\
& =\|h\|_{\ell_{1,w}(\Z^d)} \|a\|_{\ell_{1,w}(\Z^d)}
\end{align*}
where the exchange of sums is justified by Tonelli's theorem, which yields the required weighted version of Young's inequality.
Since the latter is sharp (with the choice $a=\delta[\cdot]$), it also 
shows that the induced norm of the convolution operator $a \mapsto h \ast a$ is $\|h\|_{\ell_{1,w}(\Z^d)}$.
As for the converse part, we consider the impulsive input signal $e_{\bk}=\delta[\cdot-\bk]$, which is such that
$$
\|e_{\bk}\|_{\ell_{1,w}(\Z^d)}=w[\bk] 
$$
By recalling that $e_{\bk} \ast e_{\V l}=e_{\bk+\V l}$ and invoking the Banach algebra property, we  get
\begin{align}
w[\bk+\V l]=\|e_{\bk} \ast e_{\V l}\|_{\ell_{1,w}(\Z^d)} \le \|e_{\bk} \|_{\ell_{1,w}(\Z^d)}\| e_{\V l}\|_{\ell_{1,w}(\Z^d)}=w[\bk]w[\V l],
\end{align}
which is the desired inequality for $w$ (submultiplicativity).
\end{proof}

	\subsection{Inverse-Closedness of Banach Convolution Algebras}

The issue of the inverse-closedness of weighted Banach convolution algebras was beautifully settled by Gelfand, Raikov, and Shilov \cite{Groechenig2010,Gelfand1964}.
\begin{definition}
\label{Def:GRS}
A submutiplicative weighting sequence $w[\cdot]$  fulfils the GRS (Gelfand-Raikov-Shilov) condition (in short,we say that $w[\cdot]$ is GRS) if
\begin{equation} \label{eq:GRSweight}
\lim_{m\to \infty} w[m\bk]^{1/m}=1,  \quad \forall \bk \in \Z^d.
\end{equation}
\end{definition}

The upper bound in \eqref{eq;multexpo} indicates that the faster-growing submultiplicative weights are the exponential ones.
Interestingly, these also corresponds to the breakpoint beyond which the GRS condition no longer holds.
Let $w$ be a   submultiplicative weighting sequence. If there exists $b \in[0,1)$ and $C, r, k_0> 0$ such that
$$w[\bk]\le C \ee^{r |\bk|^{b}} \mbox{ for all } |\bk|\ge k_0,$$
then $w$ satisfies the GRS condition \eqref{eq:GRSweight}.
On the other hand, $w$ is not GRS if for some $C,r,k_0 > 0$, 
$$
   C \ee^{r |\bk|}\le  w[\bk] \mbox{ for all } |\bk|\ge k_0.
$$
To see this, it suffices to consider the sequence $v[\bk]=\ee^{r|\bk|^b}$ which is such that
$
\frac{1}{m}\log v[m \bk]= r|\bk|^b m^{b-1}.
$
Clearly, the latter is GRS if and only if $\lim_{n \to \infty} n^{b-1}=0$ which is equivalent to $0\leq  b<1$.

\begin{proposition} [Weighted version of Wiener's lemma, Theorem 5.24, \cite{Groechenig2010}] 
\label{prop:Wienerweight}
Let $h \in \ell_{1,w}(\Z^d)$ where $w$ is a submultiplicative weight that satisfies the GRS condition.  If $\Fourier_{\rm d}\{h\}(\bw)=\hat h(\bw)\ne0$ for all $\bw \in \mathbb{T}^d$ (invertibility), then $g=\Fourier_{\rm d}^{-1}\{1/\hat h\} \in \ell_{1,w}(\Z^d)$.
\end{proposition}

Hence, by combining Theorem \ref{Prop:algebra1} and Proposition \ref{prop:Wienerweight}, we conclude that
$\ell_{1,w}(\Z^d)$ is an inverse-closed convolution algebra if $w$ is GRS.
Remarkably, the implication also goes the other way around, which closes the topic of discrete Banach convolution algebras. 
We summarize the situation in the following:

\begin{theorem}[Corollary 5.27, \cite{Groechenig2010}]
\label{Theo:Wienerweight}
Let $w$ be a submultiplicative weighting sequence. Then the Banach convolution algebra $\ell_{1,w} (\Z^d)$ is inverse-closed if and only if $w$ satisfies the GRS condition.
\end{theorem}

\section{Inverse-Closed Nuclear Convolution Algebras}  \label{section:EandS}

In this section, we study the sequence spaces $\Spc S(\Z^d)$ and $\Spc E(\Z^d)$. Since they are not Banach spaces, we cannot directly apply the results of Section \ref{section:review}. 
As we shall see, they are two examples of inverse-closed \emph{nuclear} convolution algebras. 
They are different in the following way: $\Spc S(\Z^d)$ is a countable \emph{intersection} of Banach spaces while $\Spc E(\Z^d)$ is a countable \emph{union} of Banach spaces. 

	\subsection{Reminder on Nuclear Spaces}\label{nuclear}

Nuclear spaces were introduced in \cite{Grothendieck1955} as a  natural complement  of Banach spaces in the field of functional analysis.
For our purpose, we are concerned by the characterization of the nuclearity of sequence spaces of a special type: countable intersections of weighted $\ell_1$ spaces and their dual spaces, for which we have the following characterization.

\begin{proposition}[Proposition 28.16, \cite{Meise1997introduction}] \label{prop:nuclearcriterion}
	Let $\bm{w} = (w_n)_{n\in \N}$ be a family of weights such that $w_{n} \leq w_{n+1}$ for every $n$. The Fr\'echet space $\Spc X = \bigcap_{n \in \N} \ell_{1,w_n} (\Z^d)$ is nuclear if and only if
	\begin{equation}
		\forall n \in \N,  \ \exists m\geq n, \quad \sum_{\bm{k}\in\Z^d} \frac{w_n[\bm{k}]}{w_m[\bm{k}]} < \infty.
	\end{equation}
\end{proposition}

This particular sequence space is studied extensively in \cite[Chapter 27]{Meise1997introduction}, where it is denoted by $\Spc X = \lambda^1(\bm{w})$. 
Here, we are also interested in spaces that are duals of nuclear Fr\'echet spaces of the form $\Spc X = \bigcap_{n \in \N} \ell_{1,w_n} (\Z^d)$.

\begin{proposition}[Theorem 9.6, \cite{Schaefer1999}] \label{prop:dualnuclear}
	The strong dual of a nuclear Fr\'echet space is nuclear.
\end{proposition}

	\subsection{The Space  $\Spc S(\Z^d)$}\label{spaceS}

We recall that $\Spc S(\Z^d)$ is the space of rapidly decreasing signals (see Definition \ref{Def:signals}).
We consider the weights $$w_n[\bk]=(1+\|\bk\|)^{n}$$
with $n \in \Z$. When $n \geq 0$,  $w_n[\cdot]$ is algebraically increasing, submultiplicative, and satisfies the GRS condition.
The definition given by \eqref{eq:S} is equivalent to  
\begin{equation} \label{eq:Sinfty}
\Spc S(\Z^d)=\bigcap_{n \in \N} \ell_{\infty,w_n}(\Z^d).
\end{equation}

For $1 \leq p \leq \infty$ and $n\in \N$, one can readily show that 
\begin{align*}
\forall a \in \ell_{p,w_n}(\Z^d),  & \quad \|a\|_{\ell_{\infty,w_n}(\Z^d)} \le  \|a\|_{\ell_{p,w_n}(\Z^d)}  \\
\forall a \in \ell_{\infty,w_n}(\Z^d), & \quad  \|a\|_{\ell_{p,w_{ \lfloor n-(d/p)-\epsilon \rfloor }}(\Z^d)} \le C \|a\|_{\ell_{\infty,w_n}(\Z^d)}  
\end{align*}
where $\epsilon >0$ and for some constant $C > 0$. 
Hence, for $1 \leq p \leq \infty$ and $n \in \N$ fixed, we have the embedding relations 

\begin{equation} \label{eq:embellp}
 \ell_{p,w_n}(\Z^d) \subseteq \ell_{\infty,w_n}(\Z^d) \subseteq \ell_{p,w_{\lfloor n-(d/p)-\epsilon \rfloor}}(\Z^d).
\end{equation}
Combining \eqref{eq:Sinfty} and \eqref{eq:embellp}, we obtain the more generic characterization
\begin{align}
\label{Eq:Sintersect}
\Spc S(\Z^d)=\bigcap_{n \in \N} \ell_{p,w_n}(\Z^d)
\end{align}
that holds for any $1\leq p \leq \infty$. 

The space $\Spc S(\Z^d)$ is a nuclear Fr\'echet space \cite[Theorem 51.5]{Treves2006}.  Considering \eqref{Eq:Sintersect} with $p=1$, we can actually apply Proposition \ref{prop:nuclearcriterion} with $m = n+d+1$ to deduce the nuclearity of $\Spc S(\Z^d)$ since
$$\sum_{\bm{k}\in\Z^d} \frac{w_n[\bm{k}]}{w_{n+d+1} [\bm{k}]} = \sum_{\bm{k}\in\Z^d} \frac{1}{(1 + \lVert \bm{k} \rVert)^{d+1}} < \infty.$$

The dual counterpart of the representation  \eqref{Eq:Sintersect}  is
$$\Spc S'(\Z^d)=\left(\Spc S(\Z^d)\right)'=\bigcup_{n \in \N} \ell'_{p,w_n}(\Z^d)=\bigcup_{n \in \N} \ell_{q,1/w_n}(\Z^d)$$
where we recall that
$\left(\ell_{p,w_n}(\Z^d)\right)'=\ell_{q,1/w_n}(\Z^d)$ with $\frac{1}{p}+\frac{1}{q}=1$ and $1\leq p<\infty$.
Since $\Spc S'(\Z^d)$ is the strong dual of a nuclear Fr\'echet space, it is nuclear as well (but not Fr\'echet), according to Proposition \ref{prop:dualnuclear}.

In view of \eqref{Eq:Sintersect} with $p=1$, we now propose to take Theorems  \ref{Prop:algebra1} and \ref{Theo:Wienerweight}  with $w=w_n$ to the limit as $n\to \infty$, which allows us to deduce the following.

\begin{theorem}
\label{S:algebra1}
The space $\Spc S (\Z^d)$ is an inverse-closed nuclear convolution algebra.
\end{theorem}
There is actually no need here to invoke the full Banach-space machinery (Young's inequality and Wiener's lemma) because there is a simpler direct proof of Theorem \ref{S:algebra1}.  Indeed, it is well known that the discrete Fourier transform of a rapidly decreasing sequence is $2 \pi$-periodic and infinitely differentiable and vice versa; \emph{i.e.}, $a \in \Spc S(\Z^d)$ if and only if $\hat a \in C^\infty(\mathbb{T}^d)$. Now, the product of two $C^\infty$ functions is
$C^\infty$ as well, so that $h \ast a \in \Spc S(\Z^d)$ if and only if $\widehat{h \ast a}\in C^\infty(\mathbb{T}^d)$. Likewise, if $\hat h \in C^\infty(\mathbb{T}^d)$ with $\hat h(\bw)\ne 0$ for all $\bw \in \mathbb{T}^d$, then $\hat g=1/\hat h \in C^\infty(\mathbb{T}^d)$, which proves the inverse-closedness.

	\subsection{The Space $\Spc E(\Z^d)$} \label{spaceE}

\paragraph{Nuclearity of $\Spc E(\Z^d)$.}
Contrary to $\Spc S(\Z^d)$, the space of exponentially decreasing sequences $\Spc E(\Z^d)$ is not a countable intersection of Banach spaces; that is, not a Fr\'echet space. However, we have the following:

\begin{proposition}
\label{E:space}
Let $1 \leq p \leq \infty$.
$\Spc E (\Z^d)$ is a nuclear space that can be specified as the inductive limit
\begin{align}
\label{Eq:Eunion}
\Spc E(\Z^d)
=\bigcup_{n \in \N } \ell_{p,v_{n}}(\Z^d)
\end{align}
with weighting sequence $v_n[\bk]=\ee^{\lvert \bk \rvert / (n+1)}$, which is exponentially increasing and submultiplicative.
\end{proposition}

\begin{proof}
We can rewrite \eqref{eq:E(Z)} as $ \Spc E(\Z^d)=\bigcup_{n \in \N\backslash \{0\}} \ell_{\infty,v_{n}}(\Z^d)$. Moreover, we have the embeddings
\begin{equation} 
	\ell_{p,v_n} (\Z^d) \subseteq \ell_{\infty,v_n} (\Z^d) \subseteq \ell_{p,v_{n+1}} (\Z^d),
\end{equation}
which shows that \eqref{Eq:Eunion} is true for every $1\leq p \leq \infty$. We consider the case $p=1$ for the rest of the proof.

We now consider  the Fr\'echet space $\Spc V(\Z^d) = \bigcap_{n \in \mathbb{N}} \ell_{1, 1/v_n}(\Z^d)$, endowed with the projective topology. We apply Proposition \ref{prop:nuclearcriterion} with $m = n+1$, for which
$$\sum_{\bm{k}\in\Z^d} \frac{1/v_n[\bm{k}]}{1 / v_{n+1} [\bm{k}]} = \sum_{\bm{k}\in\Z^d} \ee^{ - \frac{|\bm{k}|}{n(n+1)}}< \infty,$$
and the space $\Spc V(\Z^d)$ is therefore nuclear.
Since $\Spc V'(\Z^d) = \bigcup_{n \in \N\backslash \{0\}} \ell_{1,v_{n}}(\Z^d) = \Spc E(\Z^d)$, the space $\Spc E(\Z^d)$ is the strong dual of a nuclear Fr\'echet space and hence nuclear because of Proposition \ref{prop:dualnuclear}. 
\end{proof}

\paragraph{The space $\Spc E(\Z^d)$ as an inverse-closed convolution algebra.} 

\begin{theorem}
\label{E:algebra1}
The space $\Spc E(\Z^d)$ is an inverse-closed nuclear convolution algebra.
\end{theorem}

In \cite{Jaffard1990proprietes}, Jaffard established that $\Spc E(\Z^d)$ is stable by convolution (Proposition 1), and inverse-closed (Proposition 2). His work is actually more general since it covers infinite matrices that are dominated by convolution kernels (the latter being a special case of the former). The novel aspect brought by Theorem \ref{E:algebra1} is the nuclearity of $\Spc E(\Z^d)$, proved in Proposition \ref{E:space}. 
We are going to present a very short alternative proof of the inverse-closedness and algebra part of the statement for the case of convolution operators.
The enabling property, which will also play a central role in Section \ref{sec:singular}, is the following characterization of the Fourier transform of a sequence in $\Spc E(\Z^d)$.

\paragraph{Fourier transform of exponentially decreasing sequences.}

\begin{definition}
A function $f: \T^d \rightarrow \R$ is real analytic at $\bx_0$ if it can be represented by a convergent power series of the form
$$f(\bx)=\sum_{\V n\in \N^d} a_{\V n} (\bx-\bx_0)^{\V n}$$ in some open neighbourhood of $\bx_0$.
A function is real analytic on $\T^d$ if it is real analytic at every $\bm{x}_0 \in \T^d$.
\end{definition}

The space of exponentially decreasing sequence is in correspondence with the space of real analytic periodic functions. 
 
\begin{theorem}
\label{Theo:realanalytic}
A discrete signal $u[\cdot]$ is exponentially decreasing---\emph{i.e.}, $u \in \Spc E(\Z^d)$---if and only if its discrete-domain Fourier transform $\hat u$ is real analytic on $\mathbb{T}^d$.
\end{theorem}
Theorem \ref{Theo:realanalytic} is probably known but we have not been able to find a complete proof in the literature, which is the reason why we are providing our own in Appendix 1. The 1-D version of the equivalence is mentioned in 
 \cite[p. 134]{Krantz2002} as a percursor of the Paley-Wiener theorem,
and listed as an exercise in \cite[p. 27]{Katznelson:1976}.

\begin{proof}[Proof of Theorem \ref{E:algebra1}]
By Theorem \ref{Theo:realanalytic}, the algebra part of the statement is equivalent to $\hat h(\bw)  \hat a(\bw)$ and $1/\hat h(\bw)$ being analytic, which are basic properties in the theory of convergent power series. For instance, if $f$ is analytic at $\bx_0$ with $f(\bx_0)\ne0$, then
$1/f$ is analytic at $\bx_0$.  
\end{proof}

\section{Hierarchy of Inverse-closed Convolution Algebras} \label{section:main}
 
	\subsection{Classification of Convolution Algebras}
 
 \begin{definition} \label{def:types}
We consider three types of sequence spaces that are candidates for being inverse-closed convolution algebras:
\begin{itemize}
	\item \emph{Type I:} Banach spaces of the form $\ell_{1,w}(\Z^d)$ with $w$ a submultiplicative weight;
	\item \emph{Type II:} countable projective limit of Banach spaces (also called Fr\'echet spaces)
			\begin{equation} \label{eq:intersection}
				\bigcap_{n\in \N} \ell_{1,w_n}(\Z^d)
			\end{equation}
		where $(w_n)$ is a family of submultiplicative weighs such that $w_n \leq w_{n+1}$; 
	\item \emph{Type III:} countable inductive limit of Banach spaces
			\begin{equation} \label{eq:union}
				\bigcup_{n\in \N} \ell_{1,w_n}(\Z^d)
			\end{equation}
		where $(w_n)$ is a family of submultiplicative weighs such that $w_{n+1} \leq w_n$.
		 \end{itemize}
\end{definition}

Sequence spaces of types I, II, or III are convolution algebras, as intersections or unions of convolution algebras (the weights are assumed to be submultiplicative). We therefore call them \emph{convolution algebras of type I}, \emph{II}, or \emph{III}.
The typical example of convolution algebra of type II is the space $\Spc S(\Z^d)$. Moreover, we have seen that the space $\Spc E(\Z^d)$ is of type III. Note that the conditions of having increasing weights for the type II  is not restrictive: if this condition is not satisfied, it suffices to consider the family of weights $(w_0 + \ldots + w_n)_{n\in\N}$. Similar considerations can be done for the type III. 

The types specified by \eqref{eq:intersection} and \eqref{eq:union} are mutually exclusive in the following sense:

\begin{proposition}
	If $\Spc X$ is simultaneously a convolution algebra of type II and III, it is necessarily a Banach convolution algebra of type I. 
\end{proposition}

\begin{proof}
First of all, as a Fr\'echet space, it is metrizable. Being of type III, its dual is itself a Fr\'echet space, and is also metrizable. Moreover, the dual of a non-normable locally convex space is not metrizable \cite[Section 29.1]{Kothe1969topological}. Hence, if $\Spc X$ is not normable, its dual is not metrizable. This shows that $\Spc X$ is normable, and is therefore a Banach space. 
\end{proof}
 
The characterization of inverse-closed convolution algebras of type I is a classical problem that has been completely solved (see Section \ref{section:review}): the space $\ell_{1,w}(\Z^d)$ is inverse-closed if and only if it satisfies the GRS condition. We shall see now that this property also carries over naturally to the convolution algebras of type II. The case of convolution algebras of type III is more challenging and is the subject of the next section. 
 
\begin{proposition} \label{prop:GRStypeII}
	Let $\Spc X = \bigcap_{n\in \N} \ell_{1,w_n}(\Z^d)$ be a convolution algebra of type II. Then, $\Spc X$ is inverse-closed if and only if the weights $w_n$ are all GRS.
\end{proposition}

Proposition \ref{prop:GRStypeII} is the generalization of Theorem \ref{Theo:Wienerweight} from Banach to Fr\'echet sequence spaces. 

\begin{proof}
	An intersection of inverse-closed convolution algebras is inverse-closed, so the 
	GRS condition is sufficient. For the necessity, we shall assume that  one of the weights $w_{n_0}$ is not GRS, and we shall prove that $\Spc X$ is not inverse-closed. 
	
	For $w_{n_0}$ non GRS,  according to Theorem \ref{Theo:Wienerweight}, the convolution algebra $\ell_{1,w_{n_0}}(\Z^d)$ is not inverse-closed. Then, there exists a signal $a \in \ell_{1,w_{n_0}}(\Z^d)$ such that $\widehat{a}$ does not vanish, and $b = \Fourier_{\mathrm{d}}^{-1} \{ 1/\widehat{a} \} \notin  \ell_{1,w_{n_0}}(\Z^d)$. We can actually construct such a signal $a$ that is moreover compactly supported. This is shown in the proof of \cite[Corollary 5.27]{Groechenig2010}, and is also developed later in the proof of Theorem \ref{theo:GRStypeIII}. Then, the sequence $a \in \Spc X$ (because it is compactly supported), while $b = \Fourier_{\mathrm{d}}^{-1} \{ 1/\widehat{a} \}  \notin \Spc X \subset \ell_{1,w_{n_0}}(\Z^d)$. Therefore, $\Spc X$ is not inverse closed. 
\end{proof}
 
	\subsection{Inverse-closed Convolution Algebras of type III}

The case of sequence spaces of type III was addressed by Fern{\`a}ndez, Galbis and Toft  in the  general setting of convolution-dominated 
infinite-matrix algebras \cite{Fernandez2014spectral}. Their work is based on the following extension of the GRS condition. 

\begin{definition}
Consider a countable family of submultiplicative weights $\bm{w} = (w_n)_{n\in \N}$ such that $w_n \geq w_{n+1}$ for every $n$.
We say that $\bm{w}$ satisfies the \emph{extended GRS condition} if 
\begin{equation} \label{eq:GRSfamily}
	\inf_{n \in \N} \left( \lim_{m \rightarrow \infty} w_n[m \bm{k}]^{1/m}\right) = 1,  \quad \forall \bm{k} \in \Z^d.
\end{equation}
\end{definition}

When the family is reduced to one element (that is, all the weights are equal to a unique weight $w$), we recover the usual GRS condition. Moreover, if one of the   $w_n$ is GRS, then the complete family is GRS. Therefore, the scenario   that is of interest is when no member of the family is GRS, as is the case for $\Spc E(\Z^d)$. Our next result is a refinement of a theorem by 
Fern{\`a}ndez et al.\ in that it shows that the implication also goes the other way around.

\begin{theorem} \label{theo:GRStypeIII}
	Let $\bm{w} = (w_n)_{n\in \N}$ be a family of submultiplicative weights with $w_{n+1} \leq w_n$. Then, the space $\Spc X = \bigcup_{n\in\N} \ell_{1,w_n}(\Z^d)$ is inverse-closed if and only if $\bm{w}$ satisfies the extended GRS condition.
\end{theorem}

\begin{proof}
	We already know that $\Spc {X}$ is a convolution algebra. 
	According to \cite[Theorem 2.1]{Fernandez2014spectral}, a convolution algebra of type III based on a GRS family can be seen as an (uncountable) intersection of inverse-closed Banach convolution algebras, and is therefore inverse-closed itself.
	
	We should now prove that the GRS condition is necessary for the inverse-closedness. To do so, we adapt the proof of the necessity of the usual GRS condition for the inverse-closedness of \emph{Banach} convolution algebras \cite[Corollary 5.27]{Groechenig2010}. We first introduce some notations. For every $n\in \N$,  the limit $\ell_n[\bm{k}] := \lim_{m \rightarrow \infty} w_n[m \bm{k}]^{1/m}$ exists in $[1,\infty)$ for every $\bm{k} \in \Z^d$ because the $w_n$ is submultiplicative. Moreover, since the weights are decreasing, there exists for every $\bm{k} \in \Z^d$,
	$$\ell [\bm{k}] := \lim_{n} \ell_n [\bm{k}]= 	\inf_{n \in \N} \left( \lim_{m \rightarrow \infty} w_n[m \bm{k}]^{1/m}\right) \geq 1.$$
Let us now assume that the family $(w_n)$ does \emph{not} satisfy the extended GRS condition \eqref{eq:GRSfamily}.  To finish the proof, it now suffices to show that the space $\Spc X$ is {not} inverse-closed.  If the family is not GRS,  then there exists $\bm{k}_0 \in \Z^d$ such that, for every $n$, $\ell_n [\bm{k}_0]\geq \ell [\bm{k}_0]> 1$. Let us fix $\alpha>0$ such that $1\leq \mathrm{e}^{\alpha} < \ell$. In particular, it means that, for every $n\in \N$, there exists $m_0(n)$ such that,
\begin{equation} \label{eq:generalizedGRSnecessityintermediaire}
	\forall m\geq m_0(n), \ w_n[m\bm{k}_0] \geq \mathrm{e}^{\alpha m}.
\end{equation}
Consider the sequence $a [\bm{k}] = \delta[\bm{k}] - \mathrm{e}^{-\alpha} \delta[\bm{k} - \bm{k}_0]$ \footnote{We can use the same sequence $a$ for the proof of Proposition \ref{prop:GRStypeII}.}. Clearly, $a \in \Spc X$, since it is compactly supported. Moreover, the discrete Fourier transform $\widehat{a}(\bw) = 1 - \ee^{-\alpha} \ee^{-\mathrm{i} \langle \bm{\omega}, \bm{k}_0\rangle}$ satisfies
$\abs{\widehat{a} (\bm{\omega})} \geq 1 - \mathrm{e}^{-\alpha} > 0$ for every $\bm{\omega} \in \T^d$, and therefore does not vanish. The inverse $1/\widehat{a}$ of $\widehat{a}$ is given by
\begin{equation}
	\frac{1}{\widehat{a}(\bm{\omega})} = \frac{1}{1 - \mathrm{e}^{-\alpha} \mathrm{e}^{- \mathrm{i}\langle \bm{\omega}, \bm{k}_0\rangle }} = \sum_{m\geq 0} \mathrm{e}^{- m \alpha} \mathrm{e}^{- \mathrm{i} m \langle \bm{\omega}, \bm{k}_0\rangle}.
\end{equation}
This allows us to deduce that the inverse filter $b = \mathcal{F}_d^{-1} \{ 1/ \widehat{a} \}$ is given by
\begin{equation}
	b[\bm{k}] =  \begin{cases} \mathrm{e}^{-m \alpha} & \text{ if } \bm{m} = m \bm{k}_0, \\
			  0 & \text{ otherwise}. 
			  \end{cases}
\end{equation}
Therefore, we have, for every $n \in \N$,
\begin{equation*}
	\lVert b \rVert_{\ell_{1,w_n}(\Z^d)} = \sum_{\bm{k}\in \Z^d} w_n[\bm{k}] b[\bm{k}] = \sum_{m \geq 0}w_n [m \bm{k}_0]  \mathrm{e}^{-\alpha m} = \infty,
\end{equation*}
since $w_n [m \bm{k}_0]  \mathrm{e}^{-\alpha m}  \geq \mathrm{e}^{\alpha m} \mathrm{e}^{-\alpha m} = 1$ for every $m \geq m_0(n)$ due to \eqref{eq:generalizedGRSnecessityintermediaire}. This means that, for every $n$, $b \notin \ell_{1,w_n}(\Z^d)$. In doing so, we  have constructed a sequence $a \in \Spc X$ such that $\widehat{a}$ does not vanish and $b = \mathcal{F}_d^{-1} \{ 1/ \widehat{a} \} \notin \Spc X$, which proves that the convolution algebra $\Spc X$ is not inverse-closed. 
\end{proof}

	\subsection{The Case of $\Spc E(\Z^d)$: the Smallest Inverse-closed Convolution Algebra}

We shall now prove  that the space of exponentially decreasing sequence is a subspace of all the convolution algebras of Definition \ref{def:types}. We first start with a characterization of $\Spc E(\Z^d)$ as the (uncountable) projective limit of the complete family of inverse-closed Banach convolution algebras. To that end, we define $\mathcal{W}_{\rm GRS}$  as the space of all GRS submultiplicative weighting sequences.

\begin{theorem} \label{theo:icing}
Let $1\leq p \leq \infty$. 
The space $\Spc E(\Z^d)$ is the projective limit of the Banach spaces $\ell_{p,w}(\Z^d)$ for $w \in \mathcal{W}_{\rm GRS}$ \emph{i.e.},
	\begin{align}
	\label{Eq:Eintersect}
	\Spc E(\Z^d)=\bigcap_{w \in \mathcal{W}_{\rm GRS}} \ell_{p,w}(\Z^d).
	\end{align}
With $p=1$, we deduce that $\Spc E(\Z^d)$ is the projective limit of the inverse-closed Banach convolution algebras $\ell_{1,w}(\Z^d)$ with $w\in \mathcal{W}_{\rm GRS}$.
\end{theorem}

Applying  \cite[Theorem 2.1]{Fernandez2014spectral}, we already know that $\Spc E(\Z^d)$, as a sequence space of type III, can be written as $\Spc E(\Z^d) = \bigcap_{w \in \mathcal{A}} \ell_{1,w} (\Z^d)$ for some family of GRS weights $\mathcal{A} \subset \mathcal{W}_{\rm GRS}$. Our contribution is to show that $\mathcal{A} = \mathcal{W}_{\rm GRS}$ includes \textit{all} the GRS weights. Moreover, we give a self-contained proof of Theorem \ref{theo:icing} that does not rely on the results of  \cite{Fernandez2014spectral}.

We consider $\widetilde{\mathcal{W}}$ the space of weighting sequences $w \in \mathcal{W}_{\rm GRS}$ that are $\lvert \cdot \rvert$-isotropic in the sense that $w[\bk] = w[\bl]$ when $\lvert \bk \rvert = \lvert \bl \rvert$. As a first step, we show that we can restrict ourselves to $\lvert \cdot \rvert$-isotropic weighting sequences. 

\begin{lemma}
	For every $1\leq p \leq \infty$, we have 
	\begin{equation}
		\bigcap_{w \in \mathcal{W}_{\rm GRS}} \ell_{p,w}(\Z^d) = \bigcap_{w \in \mathcal{W}_{\rm GRS}} \ell_{1,w}(\Z^d) = \bigcap_{w \in \widetilde{\mathcal{W}}} \ell_{1,w}(\Z^d).
	\end{equation}
\end{lemma}

\begin{proof}
First, we remark that for every $p> 1$, $w\in \mathcal{W}_{\rm GRS}$, and $w_1 [\bk] = (1+\lVert \bk \rVert)$, one has from H\"{o}lder inequality that 
$$\lVert a  \rVert_{\ell_{1,w}(\Z^d)} \leq \lVert a  w_1 \rVert_{\ell_{p,w}(\Z^d)} \lVert 1/w_1 \rVert_{\ell_{q}(\Z^d)} =  \lVert a   \rVert_{\ell_{p,ww_1}(\Z^d)} \lVert 1/w_1 \rVert_{\ell_{q}(\Z^d)}$$
 with $1/p + 1/q = 1$ and $ww_1$ is the pointwise product between $w$ and $w_1$.
We have therefore 
$$\ell_{p,ww_1}(\Z^d) \subseteq \ell_{1,w}(\Z^d) \subseteq \ell_{p,w}(\Z^d).$$
 Since $ww_1 \in \mathcal{W}_{\rm GRS}$ for every $w\in \mathcal{W}_{\rm GRS}$, we deduce that $$\bigcap_{w \in \mathcal{W}_{\rm GRS}} \ell_{p,w}(\Z^d) = \bigcap_{w \in \mathcal{W}_{\rm GRS}} \ell_{1,w}(\Z^d).$$ for every $p< \infty$. The case $p = \infty$ is similar once we remark that $\ell_{\infty,ww_2}(\Z^d) \subseteq \ell_{1,w}(\Z^d) \subseteq \ell_{\infty,w}(\Z^d)$ with  $w_2 [\bk] = (1+\lVert \bk \rVert)^2$.
 
 For the second equality, the inclusion $\widetilde{\mathcal{W}} \subset \mathcal{W}_{\rm GRS}$ implies that $$\bigcap_{w \in \mathcal{W}_{\rm GRS}} \ell_{1,w}(\Z^d)\subseteq \bigcap_{w \in \widetilde{\mathcal{W}}} \ell_{1,w}(\Z^d).$$ For the other embedding, we remark that for any $w \in \Spc W$, there exists $\widetilde{w} \in \widetilde{\Spc W}$ such that $w \leq \widetilde{w}$. Indeed, we can simply consider $\widetilde{w} [\bm{k}] =\max_{\lvert \bm{l} \rvert = \lvert \bm{k} \rvert} w[\bm{l}]$; that is in $\widetilde{\Spc W}$. Consequently, we have $$ \bigcap_{w \in \mathcal{W}_{\rm GRS}} \ell_{1,w}(\Z^d) = \bigcap_{w \in \widetilde{\mathcal{W}}} \ell_{1,w}(\Z^d).$$
\end{proof}

\begin{proof}[Proof of Theorem \ref{theo:icing}]
First of all, a sequence $a$ is in $\Spc E(\Z^d)$ if and only if we have, for every $h\in \Spc E'(\Z^d)$, $\sum_{\bk \in \Z^d} |h[\bk] a[\bk]| <\infty$. Moreover, a sequence $(a_n)_{n\in \N}$ converges to $a$ in $\Spc E(\Z^d)$ if and only if it converges to $a$ in every space $\ell_{1,\abs{h}}(\Z^d)$. This implies that $\Spc E(\Z^d)$ is the projective limit of the spaces $\ell_{1,\abs{h}}(\Z^d)$ for $h\in \Spc E'(\Z^d)$\footnote{We say in that case that $\Spc E(\Z^d)$ and $\Spc E'(\Z^d)$ are $\alpha$-duals, see \cite[Section 30.1]{Kothe1969topological}.}; that is,
\begin{equation} \label{eq:alpha_dual}
	\Spc E(\Z^d)=\bigcap_{h \in  \Spc E'(\Z^d)} \ell_{1,\abs{h}}(\Z^d). 
\end{equation}

\textbf{Step 1.} $\widetilde{\mathcal{W}} \subset \Spc E'(\Z^d)$. 

Let $w \in \widetilde{\Spc W}$ and $\epsilon >0$. We denote by $\bm{e}_i$, $i=1\ldots d$, the canonical basis of $\R^d$. From the GRS condition, we know that, for $m \in \Z$ with $\lvert m \rvert$ big enough, $w[m \bm{e}_i]^{1/m} \leq \mathrm{e}^{\epsilon}$. Therefore, there exists $C>0$ such that for every $m\in \Z$ and $i=1,\ldots, d$, $w[m \bm{e}_i] \leq C \mathrm{e}^{\epsilon \lvert m \rvert}$. Finally, we have for every $\bk = (k_1,\ldots, k_d) \in \Z^d$, using the submultiplicativity of $w$, 
$$w[\bk] \leq w[k_1 \bm{e}_1] \cdots  w[k_d \bm{e}_d] \leq C^d \mathrm{e}^{\epsilon \lvert \bk \rvert}.$$
This holds for every $\epsilon$,  so that  $w \in \Spc E'(\Z^d)$. We deduce that $\Spc E(\Z^d) \subseteq  \bigcap_{w \in \widetilde{\mathcal{W}}} \ell_{1,w}(\Z^d)$.\\

\textbf{Step 2.} $ \bigcap_{w \in \mathcal{W}_{\rm GRS}} \ell_{1,w}(\Z^d) \subseteq \Spc E(\Z^d)$

Let $h \in \Spc E'(\Z^d)$, we shall show that there exists $w \inÊ\Spc W$ such that $\abs{h} \leq w$. In particular, this would imply that $\ell_{1,w} (\Z^d) \subseteq \ell_{1,\abs{h}}(\Z^d) $, and therefore $\bigcap_{w \in \mathcal{W}_{\rm GRS}} \ell_{1,w}(\Z^d) \subseteq \bigcap_{h \in \Spc E'(\Z^d)} \ell_{1,\abs{h}}(\Z^d)$.
Without loss of generality, we can consider that $h$ is $\lvert \cdot \rvert$-isotropic and $h\geq 1$. Indeed, we can otherwise consider the sequence $$\widetilde{h}[\bm{k}] = \max \{1, \max_{\lvert \bm{l} \rvert = \lvert \bm{k} \rvert} \abs{h[\bm{l}]} \}$$ that satisfies $\abs{h} \leq \widetilde{h}$. 

Let $f$ be the function from $\R^+$ to $\R$ that interpolates linearly the samples $(\log h [k \bm{e}_1])_{k\in \mathbb{N}}$. Since $h[\bm{0}] = 1$ and $h\geq 1$, we have $f(0) = 0$ and $f\geq 0$. Let $F$ be the least concave majorant of $f$. We have that $F(t) = \sup_{s\geq 0} \min\left (1, \frac{t}{s} \right) f(s)$ \cite{Peetre1970concave}, from which we easily deduce the following fact: if   $f(t) \leq \alpha t + \beta$ with $\alpha >0$ and $\beta \in \R$, then 
\begin{equation} \label{eq:linear_bound}
F(t) \leq \alpha t + \max(0, \beta).
\end{equation}
The sequence $h$ is in $\Spc E'(\Z^d)$. Hence, for every $\epsilon >0$, there exists a constant $C_\epsilon$ such that $h[\bk] \leq C_\epsilon \mathrm{e}^{\epsilon \lvert \bk \rvert}$. Therefore, we have, for every $\epsilon >0$ and $t\geq 0$, $f(t) \leq \log C_\epsilon + \epsilon t$. From \eqref{eq:linear_bound}, we deduce that, for every $\epsilon >0$ and $t\geq 0$,
\begin{equation} \label{eq:Fbound}
 	F(t) \leq \epsilon t + \max(0, \log C_\epsilon).
\end{equation}
Let $w$ be the sequence defined by $w[\bk] = \exp ( F(\lvert \bk \rvert))$. Then, we have $w \in \widetilde{\Spc W}$. Indeed,
\begin{itemize}
	\item $w$ is $\lvert \cdot \rvert$-isotropic by definition;
	\item $w$ satisfies the GRS condition: from \eqref{eq:Fbound}, we have $1 \leq w[m\bk]^{1/m} \leq C_\epsilon^{1/m} \mathrm{e}^{\epsilon \lvert \bk\rvert}$ for every $\epsilon>0$ and then
	$$1 \leq \liminf_{m} w[m\bk]^{1/m} \leq \limsup_{m} w[m\bk]^{1/m} \leq \mathrm{e}^{\epsilon \lvert \bk \rvert}$$
	for $\epsilon$ arbitrarely small. Hence, $w[m\bk]^{1/m} \underset{m\rightarrow \infty}{\longrightarrow} 1$;
	\item $w$ is submultiplicative: the function $F$ is concave with $F(0) = 0$. Also, it is subadditive in the sense that $F(t+s) \leq F(t) + F(s)$ for every $t,s\geq 0$. Therefore, we have that $w[\bk + \bl] \leq w[\bk]w[\bl]$ for every $\bk,\bl \inÊ\Z^d$ such that $\lvert \bk + \bl \rvert = \lvert \bk   \rvert+\lvert  \bl \rvert$. The other cases are easily deduced from the fact that $F$ is increasing and $w$ is $\lvert \cdot \rvert$-isotropic.
\end{itemize}
Finally, since $f \leq F$, we have also that $h \leq w$, which completes the proof.
\end{proof}

\begin{theorem}\label{theo:Esmallest}
	If $\Spc X$ is a convolution algebra of type I, II, or III, then we have the embedding $\Spc E(\Z^d) \subseteq \Spc X$. 
\end{theorem}

\begin{proof}
	The cases of convolution algebras of type I and II is obvious. Consider that $\Spc X$ is a convolution algebra of type III. 
	According to \cite[Theorem 2.1]{Fernandez2014spectral}, $\Spc X =   \bigcap_{w \in \mathcal{A}} \ell_{1,w} (\Z^d)$ for some family of GRS weights $\mathcal{A} \subset \mathcal{W}_{\rm GRS}$. Then,
	\begin{equation}
	\Spc E(\Z^d) = \bigcap_{w \in \mathcal{W}_{\rm GRS}} \ell_{1,w} (\Z^d)  \subseteq \bigcap_{w \in \mathcal{A}} \ell_{1,w} (\Z^d) = \Spc X.
	\end{equation}
\end{proof}

Interestingly, Theorem \ref{theo:icing} allows for an alternative proof of the nuclearity of $\Spc E(\Z^d)$, without considering its dual $\Spc E'(\Z^d)$. 
Indeed, a locally convex space defined by a family of norms is nuclear if and only if, for every norm in the family, there is a stronger norm and the inclusion operator between the associated spaces is nuclear\footnote{This characterization is sometimes chosen as the definition of nuclear spaces, see for instance \cite[Definition 50.1]{Treves2006}.}. Here, we start from $\Spc E(\Z^d)=\bigcap_{w \in \mathcal{W}_{\rm GRS}} \ell_{2,w}(\Z^d)$. We see easily that the canonical embedding from $\ell_{2,ww_2}(\Z^d)$ to $\ell_{2,w}(\Z^d)$ with $w_2[\bk]=(1+\lVert \bk \rVert)^2$ is nuclear.

In the same spirit, Theorem \ref{E:algebra1} is a consequence of Theorem \ref{theo:icing} using the Banach convolution algebra machinery: for every $w \in \mathcal{W}_{\rm GRS}$, $\ell_{1,w}(\Z^d)$ is an inverse-closed convolution algebra, so that the same holds for the intersection. 

	\subsection{The Hierarchy of Inverse-closed Convolution Algebras}
The complete situation is summarized in Figure \ref{fig:algebras} and briefly discussed below. 
\begin{figure}[h!]
\centering
\includegraphics[scale=0.55]{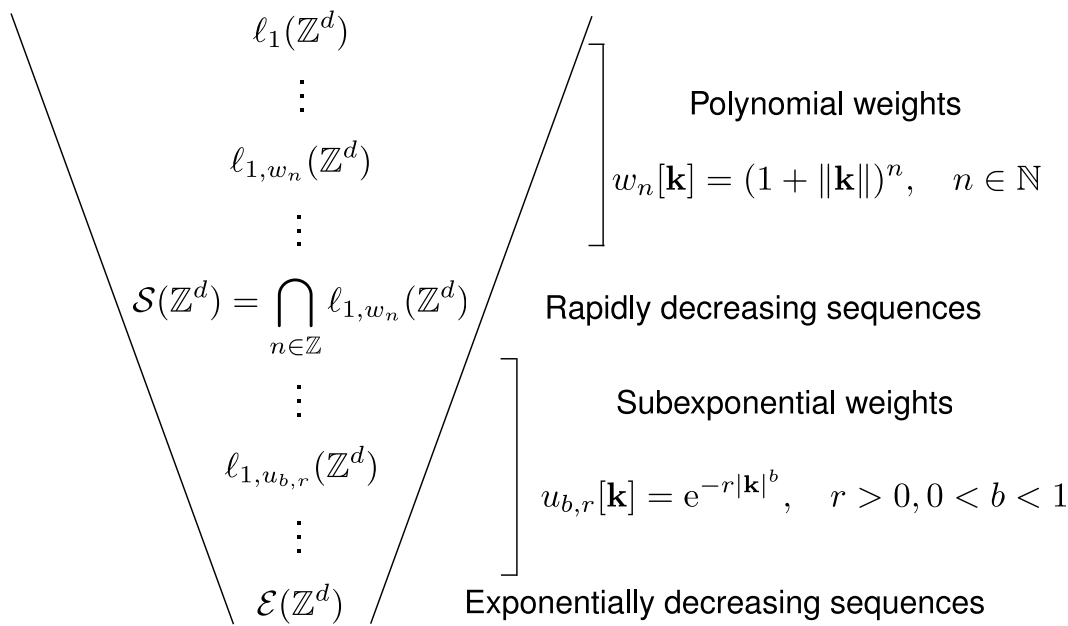}
\caption[Optional caption for list of figures]{Hierarchy of inverse-closed convolution algebras}
\label{fig:algebras}
\end{figure}

First of all, an inverse-closed convolution algebra $\Spc X$ of type I, II, or III, is such that
\begin{equation}
	\Spc E (\Z^d) \subseteq \Spc X \subseteq \ell_{1}(\Z^d).
\end{equation}
The first embedding is exactly Theorem \ref{theo:Esmallest}. The second one is a consequence of the two following facts: (i) any GRS submultiplicative weight $w$ satisfies $w\geq 1$, therefore $\ell_{1,w}(\Z^d) \subseteq \ell_1(\Z^d)$, and (ii) an inverse-closed convolution algebra of Definition \ref{def:types} is always an intersection of inverse-closed Banach convolution algebras. 

Among the inverse-closed algebra of type I, we have   algebraically decreasing sequence spaces, larger than $\Spc S(\Z^d)$, but also sequence spaces between $\Spc E(\Z^d)$ and $\Spc S(\Z^d)$, with subexponential weights, as indicated in Figure \ref{fig:algebras}. 
 
\section{Unconstrained inversion of exponentially decreasing filters} \label{sec:singular}

In 1989, De Boor, H\"ollig and Riemenschneider proved that all finite-impulse response (FIR) filters, including those that have  
frequency nulls, are invertible in $\Spc S'(\Z^d)$ \cite{DeBoor1989}. To obtain this result, they exploited the property that the frequency responses of these filters are entire functions of exponential type. We now complement our previous results by showing that this unconditional invertibility result extends to the class of filters with exponential decay.
The enabling property is the real analyticity of the Fourier transform which implies that the frequency nulls of $\widehat{h}(\bm{\omega})$, if they occur, are not too closely bunched together. For instance, the number of zeros is finite if the dimension $d=1$. 

\begin{theorem} 
\label{Theo:geninverse}
Every non-zero exponentially decreasing filter $h\in \Spc E(\Z^d)$ with (analytic) frequency response $\hat h=\Fourier_{\rm d}\{h\}$ admits a convolution 
inverse $g=\Fourier_{\rm d}^{-1}\{1/\hat h\}\in \Spc S'(\Z^d)$ 
such that $g \ast h = h \ast g = \delta[\cdot]$. 
The solutions fall into two categories:
\begin{enumerate}
\item Stable scenario where $\hat h(\bw)\ne0$ for all $\bw \in \mathbb{T}^d$: the inverse filter $g \in \Spc E(\Z^d)$ is unique and exponentially decreasing. %

\item Singular scenario where $\hat h(\bw)$ is vanishing at some frequencies $(\bw_i)$: the inverse filter $g=\Fourier_{\rm d}^{-1}\{1/\hat h\} \in \Spc S'(\Z^d)$ is of slow growth, meaning that there exists an integer $n \in \N$ and a constant $C=\|g\|_{\ell_{\infty,1/w_n}(\Z^d)}$ such that
$$
\left| g[\bk]\right| \le C (1 + \|\bk\|)^n .
$$
\end{enumerate}
\end{theorem}
\begin{proof} 
The first statement is the second part of Theorem \ref{E:algebra1}.
For the second more involved part, we first observe that $\hat h(\bw)=\sum_{\bk \in\Z^d} h[\bk] \ee^{-\mathrm{i} \langle \bw, \bk\rangle}$ is also analytic over $\R^d$ since it is $2\pi$-periodic. We then rely on a powerful theorem in distribution theory that settled the division problem initially raised by Laurent Schwartz \cite{Atiyah1970,Lojasiewicz1959}.
\begin{theorem}[{\L}ojasiewicz's division theorem]
Let $\hat h$ be real analytic (and non-identically zero) over $\R^d$. Then, the equation $\hat h \hat g = \hat u$ admits a distributional solution 
$\hat g \in \Spc S'(\R^d)$ (Schwartz' space of tempered distributions over $\R^d$)
for any tempered distribution $\hat u$. 
\end{theorem} 
This allows us to deduce that 
$\hat g=1/\hat h \in \Spc S'(\R^d)$. 
Since $\hat g$ is periodic as well, we have that  $\hat g \in \Spc S'(\mathbb{T}^d)$ so that $g \in \Spc S'(\Z^d)$, based on the standard property that the Fourier coefficients of a periodic distribution are slowly-increasing.
\end{proof}
\section{Application to cardinal spline interpolation} \label{sec:application}

Given a series of data points $(f[\bk])_{\bk \inZ^d}$, the classical  problem of cardinal spline interpolation is
to determine a function $f: \R^d\to\R$ within some given spline space such that
$\left.f(\V x)\right|_{\V x=\V k}=f[\V k]$ for all $\V k\in \Z^d$ (interpolation condition) \cite{Schoenberg1973}.
In Schoenberg's classical formulation, polynomial splines 
are represented as linear combinations of basis functions that are integer shifts of a compactly-supported B-spline $\varphi: \R^d \to \R$:
\begin{align}
\label{Eq:splineexpand}
f(\bx)=\sum_{\V k \in \Z^d} c[\V k] \varphi(\bx-\V k).
\end{align}
The determination of the spline interpolant of $f[\cdot]$ then reduces to
solving the discrete convolution equation $f[\bk]=\left.f(\V x)\right|_{\V x=\V k}=\sum_{\V l \in \Z^d} c[\V l] \varphi(\V k-\V l)=(\varphi[\cdot]\ast c)[\bk]$ where
$\varphi[\bk]=\left.\varphi(\V x)\right|_{\V x=\V k}$ is the sampled version of $\varphi$. Under the assumption that the system is invertible, the solution is given by \eqref{Eq:splineexpand} with
$c[\bk]=(h \ast f)[\bk]$ where $h$ is the discrete convolution inverse of $\varphi[\cdot]$ \cite{Th'evenaz2000}.
Equivalently, we have that
\begin{align}
\label{Eq:Intexpand}
f(\V x)=\sum_{\bk \in \Z^d} f[\bk] \varphi_{\rm int}(\bx-\bk)
\end{align}
where the interpolation kernel (or Lagrange function)
\begin{align}
\label{Eq:Interpolant}
\varphi_{\rm int}(\bx)=\sum_{\bk \in \Z^d}h[\bk] \varphi(\bx-\bk)
\end{align}
is the unique cardinal spline that interpolates the kronecker delta sequence $\delta[\cdot]$. Eq.\ \eqref{Eq:Intexpand} expresses the one-to-one relation between
the spline interpolant $f: \R^d \to \R$ and its sample values $f[\cdot]$, with the Lagrange function $\varphi_{\rm int}$ providing the mathematical description of the interpolation algorithm.
A key descriptor is the rate of decay of $\varphi_{\rm int}$, as it tells us the influence of neighboring samples on the value of the function at some non-integer location $\V x_0 \inR^d$.

We shall now use our results on convolution algebras to infer the decay of the Lagrange function, $\varphi_{\rm int}$, from the properties of $\varphi$.
First, we shall consider the case of localized generators (e.g., $\varphi \in L_1(\R^d)$) which will allow us to recover the classical results on the exponential decay of polynomial spline interpolants and their higher dimensional variants \cite{Chui1987,Schoenberg1973}.
To encompass an even broader class of splines, we shall then extend the formulation to the case where the generator is the slowly increasing Green's function of some differential operator $\Lop$ \cite{Micchelli1976}. For instance, in dimension $d=1$, cubic splines may be described as $\Lop$-splines with $\Lop=\Dop^4$ being the 
fourth derivative operator, meaning that they admit an expansion of the form
$$
f(x)=
\sum_{k\inZ} c[k] \tfrac{(x-k)_+^3}{3!}
$$
where the generator $\varphi(x)=\frac{x_+^3}{3!}$ (one-sided cubic monomial) is the causal Green's function of $\Dop^4$.
The challenge there is the lack of decay of $\varphi$ which calls for a more sophisticated treatment. This scenario will also allow us to illustrate 
the use of Theorem \ref{Theo:geninverse} on the inverse of convolution operators with frequency nulls.
%
%

%
%
%

\subsection{Integrable basis functions}
\label{Sec:splineI}

Starting from a continuous generator $\varphi \in L_1(\R^d)$, the problem is to specify the interpolator for the shift-invariant space
$$
V_\varphi=\left\{f(\bx)=\sum_{\bk \inR^d} c[\bk] \varphi (\bx -\bk): c \in \ell_2(\Z^d) \right\}.
$$
As explained before, the Lagrange function $\varphi_{\rm int}$---\emph{i.e.}, the unique member of $V_\varphi$ such that
$\varphi_{\rm int}(\V k)=\delta[\V k]$ for every $\V k \inZ^d$---is given by \eqref{Eq:Interpolant} where $h[\cdot]$ is the convolution inverse of
$\varphi[\cdot]$.
In the Fourier domain, this yields
\begin{equation}\label{eq:interp_fourier_L1}
\widehat h(\bw)
=
\frac{1}{\sum_{\bk \inZ^d} 
\varphi(\bk) \ee^{-\jj \langle \bw, \bk\rangle}}
= 
\frac{1}{\sum_{\V n \inZ^d} \widehat \varphi(\bw - 2 \pi \V n)},
\end{equation}
where $\widehat \varphi(\bw)=\int_{\R^d} \varphi(\bx)\ee^{-\jj \langle \bw, \bx\rangle}\dint \bx$ is the continuous-domain Fourier transform of $\varphi$. This is well-defined with some minor assumptions on $\varphi$, while we also assume that the denominator of \eqref{eq:interp_fourier_L1} is non-vanishing. The second equality follows from Poisson's summation formula.

Our results allow one to transfer, under mild conditions, the decay property of an integrable basis function $\varphi$ to its associated interpolator $\varphi_{\mathrm{int}}$.  To do so, we propose to introduce the continuous-domain counterparts of the sequence spaces in Figure \ref{fig:algebras}. 
They are closely related to weighted Wiener amalgam spaces as 
Wiener amalgam spaces that can be tracked back to \cite{wiener1926representation,Wiener1932}. They allow to amalgam local and global criteria on functions \cite{feichtinger1991wiener,heil2003introduction} and have strong connection with sampling theory \cite{aldroubi1998exact,feichtinger1990new}. This is also the spirit of the following definition. 
Let $w$ be a submultiplicative weight function. Then, we set
\begin{equation}
	W_{1,\infty,w}(\R^d) 
	= 
	\left\{ 
	f : \R^d \rightarrow \R : 
	\lVert f \rVert_{W_{1,\infty,w}(\R^d)} 
	= 
	\sup_{\V x \in [0,1]^d}  
	\sum_{\V k \in \Z^d}
	\lvert f ( \V x + \V k) \rvert 
	w(\V k) < \infty 
	\right\},
	\end{equation} which is a Banach space for the norm $\lVert f \rVert_{W_{1,\infty,w}(\R^d)}$. 
This means that the sequences $f_{\V x_0} [\cdot]$ defined by $f_{\V x_0} [\V k] = f(\V x_0 + \V k)$ is all in $\ell_{1,w}(\Z^d)$ for every $\V x_0 \in [0,1]^d$, and that their $\ell_{1,w}$-norms are uniformly bounded. 
We also define, in line with \eqref{Eq:Sintersect} and \eqref{Eq:Eintersect} (for $p=1$), the space of rapidly decaying functions
\begin{equation}
	\mathcal{R}(\R^d) = \bigcap_{n\in \N} W_{1,\infty,w_n}(\R^d),
\end{equation}
together with the space of exponentially decaying ones
\begin{equation} \label{eq:exponentialdecayfunction}
	\mathcal{E}(\R^d) = \bigcap_{w \in \mathcal{W}_{\mathrm{GRS}}} W_{1,\infty, w}(\R^d).
\end{equation}
We recall that the weights $w_n$ impose algebraic decay while the GRS weights further enforce subexponential decay. The inclusion of 
$f \in\mathcal{E}(\R^d)$ corresponds to the notion of exponentially decay in direct analogy with the discrete case in Theorem \ref{theo:icing}. 
We start with a preliminary result.

\begin{lemma}\label{lemma:decayofint}
	Let $w$ be a submultiplicative weight.
	Then, for any function $\varphi \in W_{1,\infty,w}(\R^d)$ and any sequence $a \in \ell_{1,w}(\Z^d)$, the function $\psi = \sum_{\V k} a[\V k] \varphi (\cdot - \V k)$ is in $W_{1,\infty,w}(\R^d)$
\end{lemma}
\begin{proof}
	Fix $\V x_0 \in [0,1]^d$ and set $\varphi_{\V x_0} [\V k] = \varphi(\V x_0 + \V k)$ (idem for $\psi_{\V x_0}[\cdot]$), which allows us to write $\psi_{\V x_0}[\cdot] = (\varphi_{\V x_0} [\cdot] * a)$ as a discrete convolution. We then invoke Young's inequality (valid because $w$ is submultiplicative) to deduce that
	\begin{align}
		\lVert \psi_{\V x_0}[\cdot] \rVert_{\ell_{1,w}(\Z^d)} &\leq \lVert \varphi_{\V x_0}[\cdot] \rVert_{\ell_{1,w}(\Z^d)}\lVert a\rVert_{\ell_{1,w}(\Z^d)} \nonumber \\
		&\leq \lVert \varphi \rVert_{W_{1,\infty,w}(\R^d)}\lVert a\rVert_{\ell_{1,w}(\Z^d)}.
	\end{align}
The last bound, which is independent of $\V x_0$, then yields
	$$ \lVert \psi \rVert_{W_{1,\infty,w}(\R^d)} = \sup_{\V x_0 \in [0,1]^d} \lVert \psi_{\V x_0}[\cdot] \rVert_{\ell_{1,w}(\Z^d)} < \infty.$$
\end{proof}

\begin{proposition}\label{prop:lagrangedecay}
	Let $\varphi \in L_1(\R^d)$ be a function such that $\sum_{\V n \in \Z^d} \widehat{\varphi} ( \V \omega - 2 \pi \V n)$ does not vanish for any $\V \omega$. 
	If $\varphi$ has algebraic, fast, or exponential decay (in the sense given above), then the same holds true for the Lagrange function $\varphi_{\mathrm{int}}$ defined by \eqref{Eq:Interpolant}. 
\end{proposition}

\begin{proof}
	Assume that $\varphi \in W_{1,\infty,w}(\R^d)$ for some submultiplicative GRS weight $w$. 
	The weakest assumption is $\varphi \in W_{1,\infty,w}(\R^d)$ with $w$ being the constant function 1. This implies that $\varphi[\cdot]\in\ell_1 (\Z^d)$. It follows from the Poisson summation formula that the expression
	\begin{equation}
	\sum_{\V n \in \Z^d} \widehat{\varphi} (\V \omega - 2\pi\V n)
	=
	\sum_{\V k \in \Z^d} \varphi(\V k ) \ee^{-\jj \langle \bw, \V k \rangle}
	\end{equation}
	is continuous and hence also bounded on $\mathbb{R}^d$.
	
	The hypothesis that $\sum_{\V n \in \Z^d} \widehat{\varphi} (\V \omega - 2\pi\V n)$ does not vanish implies that the sequence $\varphi[\cdot]$ is $\ell_2$-invertible, its inverse being denoted by $h$. Because $\ell_{1,w}(\Z^d)$ is inverse-closed (Theorem \ref{Theo:Wienerweight}), we know that $h \in \ell_{1,w}(\Z^d)$.
	Then, from Lemma \ref{lemma:decayofint}, we deduce that $\varphi_{\mathrm{int}} = \sum_{\V k \in \Z^d} h[\V k] \varphi (\cdot - \V k) \in W_{1,\infty,w}(\R^d)$, which allows us to transfer the decay of $\varphi$ to the Lagrange function $\varphi_{\mathrm{int}}$:
	\begin{itemize}
	\item Algebraic decay: $\varphi \in W_{1,\infty,w_n}(\R^d)$ for some $n \geq 0$ $\Rightarrow \varphi_{\mathrm{int}}\in W_{1,\infty,w_n}(\R^d)$.
	\item Fast decay: $\varphi \in \mathcal{R}(\R^d) \Leftrightarrow \varphi \in W_{1,\infty,w_n}(\R^d)$ for every $n\geq 0$ $\Rightarrow \varphi_{\mathrm{int}} \in W_{1,\infty,w_n}(\R^d)$ for every $n$, where the latter is equivalent to $\varphi_{\mathrm{int}} \in \mathcal{R}(\R^d)$.
	\item Exponential decay: $\varphi \in \mathcal{E}(\R^d) \Rightarrow\varphi \in W_{1,\infty,w}(\R^d)$ for every  GRS weight. Then, the previous argument implies that $h$ (the inverse filter of $\varphi[\cdot]$) is in $\ell_{1,w}(\Z^d)$ for every GRS weight $w$. According to Theorem \ref{theo:icing}, this is equivalent to $h \in \mathcal{E}(\Z^d)$. Again, we deduce that $\varphi_{\mathrm{int}}  \in W_{1,\infty,w}(\R^d)$ for every $w$, which is equivalent to $\varphi_{\mathrm{int}} \in \mathcal{E}(\R^d)$.
	\end{itemize}
\end{proof}

\subsection{Slowly increasing basis functions}

We now generalize the decay estimates to account for a larger class of functions. More precisely, we consider generators $\varphi$ of slow growth---typically, the Green's function of some differential operator $\Lop$ with $\widehat{\varphi}(\bw)= 1/\widehat{L}(\bw)$. Formally, we can still define the interpolator as in the previous case. The corresponding Fourier-domain representation is
\begin{equation}\label{eq:phi_int_fourier}
	\widehat{\varphi}_{\text{int}}(\bm{\omega})
	=
	\frac{\widehat{\varphi}(\bm{\omega})}
	{\sum_{\bm{n}\in \mathbb{Z}^d}\widehat{\varphi}(\bm{\omega}- 2\pi \bm{n})}= \widehat h(\bw) \widehat{\varphi}(\bm{\omega}) 
\end{equation}
where 
$\widehat h(\bw)$ is also specified by \eqref{eq:interp_fourier_L1}.

Under some restrictions on $\widehat{\varphi}$, we shall prove that the interpolator exhibits exponential decay. We shall also show that the spline generator $\varphi$ can be recovered from $\varphi_{\text{int}}$ by the formula
\begin{equation}\label{eq:repro_green}
	\varphi(\bm{x})
	=
	\sum_{\bm{k}\in \mathbb{Z}^d} p [\bm{k}] \varphi_{\text{int}}(\bm{x}- \bm{k})
\end{equation}
where $p$ is a slowly increasing sequence.

\begin{definition}
We view $\mathbb{R}^d$ and $\mathbb{T}^d$ as living inside $\mathbb{C}^d$, and we consider the following tubes:
\begin{itemize}
\item 
${\displaystyle
	\mathbb{T}^d
	=
	\setb{\bm{x} = (x_1,\dots,x_d)\in \mathbb{R}^d}{ -\pi < x_j \leq \pi, j=1\dots d} 
}$;

\item
${\displaystyle
	\mathbb{T}_{\epsilon}^d 
	=
	\setb{\bm{x} 
	= 
	(x_1,\dots,x_d)\in \mathbb{C}^d}
	{\Re(\bm{x})\in \mathbb{T}^d, -\epsilon < \Im(x_j) \leq \epsilon, j=1\dots d };
}$
\item
${\displaystyle
	\mathbb{R}_{\epsilon}^d 
	=
	\setb{\bm{x} = (x_1,\dots,x_d)\in \mathbb{C}^d}{\Re(\bm{x})\in \mathbb{R}^d, -\epsilon < \Im(x_j) \leq \epsilon, j=1\dots d }.
}$
\end{itemize}
\end{definition}

Before stating our theorem on the exponential decay of interpolators, we begin with a lemma 
to show that the sequence 
\begin{equation}
	h[\cdot]=\mathcal{F}^{-1}_{\rm d}\set{  \frac{1}
	{\sum_{\bm{n}\in \mathbb{Z}^d}\widehat{\varphi}(\bm{\omega}- 2\pi \bm{n})}  }
\end{equation}	
is well defined and decays exponentially.

\begin{lemma}\label{lem:analytic_interp_1}

Suppose we have a continuous, slowly increasing $\varphi: \mathbb{R}^d \rightarrow \mathbb{R}$ that has a generalized Fourier transform
$\widehat{\varphi}:\mathbb{R}^d\rightarrow [0,\infty)$ that satisfies the following conditions:
\begin{itemize}
\item $1/\widehat{\varphi}$ has an analytic extension to $\mathbb{R}_{\epsilon}^d$ for some $\epsilon>0$. This implies that $1/\widehat{\varphi}$ is real analytic on $\mathbb{R}^d$;
\item $\widehat{\varphi}(\bm{\omega}) \geq 0$ for $\bm{\omega} \in \mathbb{R}^d$;
\item $\exists \gamma>0$ and $C>0$ such that 
$\widehat{\varphi}(\bm{\omega})
<
C \abs{\bm{\omega}}^{-d-\gamma}$ 
for 
$\bm{\omega} \in \mathbb{R}_{\epsilon}^d 
\backslash \mathbb{T}_{\epsilon}^d$.

\end{itemize}
Then the function 
\begin{equation}
	\widehat{h}(\bm{\omega})
	:=
	\frac{1}{\sum_{\bm{n}\in\mathbb{Z}^{d}} 
	\widehat{\varphi}(\bm{\omega}- 2\pi \bm{n})} 
\end{equation}
is analytic on $\mathbb{R}_{\epsilon^{\prime}}^d$ for some $\epsilon\geq\epsilon^{\prime}>0$.

\end{lemma}
\begin{proof}
The idea for our proof comes from \cite{madych1990}, where the authors showed that polyharmonic interpolators have exponential decay. First, we rewrite the expression as 
\begin{align*}
	\widehat{h}(\bm{\omega})
	&=
	\frac{1}{\widehat{\varphi}(\bm{\omega})
	+ 
	\sum_{\bm{n} \neq \bm{0}} 
	\widehat{\varphi}(\bm{\omega}- 2\pi \bm{n})} 		=
	\frac{1}{\widehat{\varphi}(\bm{\omega})}\left(
	\frac{1}{1 + \widehat{\varphi}(\bm{\omega})^{-1} 
	F(\bm{\omega})}\right)
%
\end{align*}
where
\begin{equation}
	F(\bm{\omega})
	:=
	\sum_{\bm{n} \neq \bm{0}} \widehat{\varphi}(\bm{\omega}- 2\pi \bm{n}).
\end{equation}
Now, for $\bm{n}\neq \bm{0}$, $1/\widehat{\varphi}(\cdot- 2\pi \bm{n})$ is an analytic function that is bounded away from $0$ on $\mathbb{T}_\epsilon^d$. 
Hence its inverse $\widehat{\varphi}(\cdot- 2\pi \bm{n})$ is also analytic on $\mathbb{T}_\epsilon^d$.
The decay of  $\widehat{\varphi}$ ensures that the series in $F$ converges uniformly on compact subsets of $\mathbb{T}_\epsilon^d$. 
Therefore, the series converges to an analytic function.  
Also, the analyticy of $1/\widehat{\varphi}$ implies that the product $\widehat{\varphi}^{-1}F $ is analytic. 
Then we choose $\epsilon \geq\epsilon^{\prime}>0$ so that $1+\widehat{\varphi}^{-1}F $ is bounded away from $\bm{0}$ on $\mathbb{T}_{\epsilon^{\prime}}^d$.  
This allows us to invert with the guarantee that $(1+\widehat{\varphi}^{-1}F )^{-1}$ is analytic on $\mathbb{T}_{\epsilon^{\prime}}^d$. 
We then multiply by the analytic function $1/\widehat{\varphi}$ and find that $\widehat{h}$ is analytic on $\mathbb{T}_{\epsilon^{\prime}}^d$. 
Finally, $\widehat{h}$ extends periodically to an  analytic function on $\mathbb{R}_{\epsilon^{\prime}}^d$.
\end{proof}

\begin{theorem}
If $\varphi$ is a basis function as defined in Lemma \ref{lem:analytic_interp_1}, then the corresponding interpolating function 
$\varphi_{\rm{int}}:\mathbb{R}^d\rightarrow \mathbb{C}$, defined in the Fourier domain by \eqref{eq:phi_int_fourier}, decays exponentially fast.
\end{theorem}
\begin{proof}
We rewrite the Fourier transform of $\varphi_{\text{int}}$ as
\begin{equation}
	\widehat{\varphi}_{\text{int}}(\bm{\omega})
	=
	\frac{1}{1 + \widehat{\varphi}(\bm{\omega})^{-1} 
	F(\bm{\omega})},
\end{equation}
which was found to be analytic on some $\mathbb{T}_{\epsilon^{\prime}}^d$ in the proof of Lemma \ref{lem:analytic_interp_1}. 

Also, since $1/\widehat{\varphi}$ is bounded away from $0$ on the boundary of $\mathbb{T}_{\epsilon}^d$, 
$\widehat{\varphi}$ is analytic on a region $R$ that contains $\mathbb{R}_{\epsilon^{\prime}}^d \backslash \mathbb{T}_{\epsilon^{\prime}}^d$ and nontrivially intersects $\mathbb{T}_{\epsilon^{\prime}}^d$. Since $\widehat{h}$  is analytic on $\mathbb{R}_{\epsilon^{\prime}}^d$, the product $\widehat{\varphi}_{\text{int}}=\widehat{h}\widehat{\varphi}$ is also analytic on $R$.  

These two facts together imply that $\widehat{\varphi}_{\text{int}}$ is analytic on $\mathbb{R}_{\epsilon^{\prime}}^d$.  
Moreover, the decay of $\widehat{\varphi}$ ensures that on $\mathbb{R}_{\epsilon^{\prime}}^d$, 
$\widehat{\varphi}_{\text{int}}$ is integrable over translates of $\mathbb{R}^d$.  Hence $\widehat{\varphi}_{\text{int}}$ is analytic on $\mathbb{R}_{\epsilon^{\prime}}^d$ and therefore its inverse Fourier transform (on $\mathbb{R}^d$) has exponential decay \cite[Theorem 9.14]{Reed1980methods}.
\end{proof}

\begin{corollary}
The reproduction property of \eqref{eq:repro_green} is valid and the sequence $p$ has at most polynomial growth.
\end{corollary}
\begin{proof}
In Lemma \ref{lem:analytic_interp_1} we showed that the sequence $h$ of \eqref{Eq:Interpolant} is exponentially decaying.  Therefore Theorem \ref{Theo:geninverse} implies that its convolution inverse $p$ is of slow growth.
\end{proof}

Let us now apply those results to the case of the cubic splines with $\Lop=\Dop^4$.
These splines can be generated from the canonical Green's function of $\Dop^4$
whose Fourier transform is
$$
\widehat{\varphi}(\omega)=\frac{1}{\widehat{L}(\omega)}=\frac{1}{(\jj \omega)^4},
$$
which satisfies the assumptions of Lemma \ref{lem:analytic_interp_1}. This allows us to deduce that the corresponding Lagrange function $\varphi_{\rm int, 3}$ has exponential decay, while Corollary 1 ensures the validity of the reproduction formulas
$$
x_+^3=\sum_{k\in \N} k^3\varphi_{\rm int,3}(x-k) \quad \mbox{ and } \quad |x|^3=\sum_{k\in \Z} |k|^3\varphi_{\rm int,3}(x-k).
$$
The same considerations apply to the whole class of exponential splines with generic $\Lop=P(\Dop)=\Dop^N + a_{N-1} \Dop^{N-1} \cdots + a_0 \Op I$ \cite{Unser2005}, or, in higher dimensions, for the extensions of the polyharmonic splines where $\Lop=P(-\Delta)$ is a suitable polynomial of the (negative) Laplacian operator $-\Delta$ with Fourier symbol $\|\bw\|^2$.

\appendix
\section{Proof of Theorem \ref{Theo:realanalytic}}

We show that a periodic function on $\mathbb{T}^d$ is real-analytic if and only if its Fourier coefficients are exponentially decreasing.
We use the multidimensional version of a classical result in the theory of real analytical functions  \cite[Proposition 1.2.10]{Krantz2002}.

\begin{proposition}
\label{Prop:realanal}
A periodic function $\hat f$ is real analytic on $\mathbb{T}^d$ if and only if  it is infinitely differentiable on $\mathbb{T}^d$---\emph{i.e.}, $\hat f \in C^\infty(\mathbb{T}^d)$---and there exist two constants $C,R>0$ such that \begin{align}
\label{eq:derivdecay}
\sup_{\bm{\omega} \in \mathbb{T}^d} \left| \partial^{\V n} \hat f(\bw)\right| \le 
C \frac{\V n!}{R^{\abs{\bm{n}}}}
\end{align}
for all multi-indices $\V n \in \N^d$.
\end{proposition}

We shall use the following:
\begin{lemma}\label{lem:abel_1d}
Let $c>0$. There are constants $M,R>0$ such that for every $n \in \mathbb{N}$.
\begin{equation}
	\sum_{k=0}^{\infty} k^n e^{-ck}
	\leq
	M\frac{n!}{R^n}.
\end{equation}
\end{lemma}
\begin{proof}
We prove the result by induction on $n$.
For $n=0$, one has $\sum_{k=0}^{\infty} e^{-ck}
	=
	\frac{e^c}{e^c-1},$
so the result is true as long as $M>\frac{e^c}{e^c-1}$. 
Let $n\geq 1$.
We assume the result holds for all $m<n$. Then, applying Abel's lemma for summation by parts, one has
\begin{align*}
	\sum_{k=0}^{\infty} 
	k^n e^{-ck} 
	&=
	\frac{-1}{1-e^{-c}}
	\sum_{k=0}^{\infty} 
	k^n 
	\parenth{
	e^{-c(n+1)} - e^{-cn}
	} \\
	&=
	\frac{1}{1-e^{-c}}
	\sum_{k=0}^{\infty}
	\parenth{
	(k+1)^n - k^n
	}
	e^{-c(k+1)} \\
	&=
	\frac{1}{1-e^{-c}}
	\sum_{k=0}^{\infty}
	\parenth{
	\sum_{m=0}^n
	\binom{n}{m} k^m - k^n
	}
	e^{-c(k+1)} \\
	&=
	\frac{e^{-c}}{1-e^{-c}}
	\sum_{k=0}^{\infty}
	\sum_{m=0}^{n-1}
	\binom{n}{m} k^m 
	e^{-ck}.
\end{align*}
Changing the order of summation and applying our assumption gives
\begin{align*}
	\sum_{k=0}^{\infty} 
	k^n e^{-ck} 
	&=
	\frac{e^{-c}}{1-e^{-c}}
	\sum_{m=0}^{n-1}
	\binom{n}{m}
	\sum_{k=0}^{\infty}
	 k^m 
	e^{-cn} \\
	&\leq
	\frac{M e^{-c}}{1-e^{-c}}
	\sum_{m=0}^{n-1}
	\binom{n}{m}
	\frac{m!}{R^m}.
\end{align*}
Under the assumptions that $R<1$ and $\frac{e^{-c}}{1-e^{-c}}
	\sum_{k=0}^{\infty}
	\frac{1}{k!}
	<
	\frac{1}{R}$, we have
\begin{align*}
	\sum_{k=0}^{\infty} 
	k^n e^{-ck} 
	&\leq
	\frac{M}{R^{n-1}}
	\frac{e^{-c}}{1-e^{-c}}
	\sum_{m=0}^{n-1}
	\binom{n}{m}
	m!\\
	&=
	\frac{Mn!}{R^{n-1}}
	\frac{e^{-c}}{1-e^{-c}}
	\sum_{m=0}^{n-1}
	\frac{1}{(n-m)!} \\
	&\leq
	\frac{Mn!}{R^{n-1}}
	\frac{e^{-c}}{1-e^{-c}}
	\sum_{k=0}^{\infty}
	\frac{1}{k!} \\
	&\leq 
	\frac{Mn!}{R^{n}}.
\end{align*}
This completes the induction step and the proof.
\end{proof}

\begin{proof}[Proof of Theorem \ref{Theo:realanalytic}]
Let $u[\cdot] \in \Spc S'(\Z^d)$ and $\hat{u}$ its discrete Fourier transform. For convenience and without loss of generality, we assume that $u[\bm{0}]=0$. 

(i) We assume that $u[\cdot] \in \Spc E(\Z^d)$, meaning that there exist constant $C,r>0$ such that $\abs{u[\bk]} \leq C\mathrm{e}^{-r\lvert \bk\rvert}$.
Then,  for every $\bm{n} \in \N^d$ and $\bm{\omega} \in \mathbb{T}^d$, we have
$$ 	\hat{u}^{(\bm{n})}(\bm{\omega})
	=
	\sum_{\bk\in \Z^d}  \mathrm{j}^{\abs{\bm{n}}} \bk^{\bm{n}} u[\bk] \mathrm{e}^{ \mathrm{j} \langle \bk , \bm{\omega} \rangle}
.$$ We have therefore, 
\begin{align} \label{eq:bounduhat}
	\sup_{\bm{\omega} \in \mathbb{T}^d} \abs{\hat{u}^{(\bm{n})}(\bm{\omega})} 
	&\leq	\sum_{\bk \in \Z^d} \abs{\bk^{\bm{n}}} \abs{ u[\bk]} \nonumber \\
	&{\leq} 	 C\sum_{\bk \in \Z^d}   \abs{\bk^{\bm{n}}} \mathrm{e}^{-r \lvert\bk \rvert} \nonumber \\
	& =  		C2^d  \prod_{j=1}^d \sum_{k_j\in \N} k_j^{n_j}    \mathrm{e}^{- r k_j} \nonumber \\
	& {\leq}	C 2^d M^d \prod_{j =1}^d \frac{n_j !}{R^{n_j}} \nonumber \\
	& = 		(C M^d 2^d) \frac{\bm{n}!}{R^{\abs{\bm{n}}}},
\end{align}
where the constants $M$ and $R$ come from Lemma \ref{lem:abel_1d}. According to Proposition \ref{Prop:realanal}, \eqref{eq:bounduhat} implies that $\hat{u}$ is real analytic on $\T^d$. 

(ii) Let now assume that $\hat{u}$ is real analytic on $\T^d$. From Proposition \ref{Prop:realanal}, there exists constants $C,R>0$ such that $\abs{\bm{k}^{\bm{n}} u[\bm{k}]} \leq \lVert \hat{u}^{(\bm{n})} \rVert_{L_2(\T^d)} \leq C \frac{\bm{n} !}{R^{\abs{\bm{n}}}} \leq C \frac{\bm{n}^{\bm{n}}}{R^{\abs{\bm{n}}}}$. 
Therefore, for $\widetilde{C}$ big enough, one has, for every $\bm{k}\in \Z^d$ and $\bm{n} \in \N^d$,
\begin{equation}
	\abs{u[\bm{k}] } \leq \widetilde{C} \frac{\bm{n}^{\bm{n}}}{1 + (R \bm{k})^{\bm{n}} }.
\end{equation}
Indeed, 
if $ (R \bm{k})^{\bm{n}} \leq 1$, because $\frac{1}{2} \leq \frac{1}{1 +  (R \bm{k})^{\bm{n}}}$, we have $\abs{u[\bm{k}] } \leq \lVert \hat{u} \rVert_{L_2(\T^d)} \leq 2 \lVert \hat{u} \rVert_{L_2(\T^d)}  \frac{\bm{n}^{\bm{n}}}{1 + (R \bm{k})^{\bm{n}}}$, and 
if $ (R \bm{k})^{\bm{n}} > 1$, then $\frac{1}{(R \bm{k})^{\bm{n}}} \leq \frac{2}{1 + (R \bm{k})^{\bm{n}}}$, and $\abs{u[\bm{k}] } \leq C\frac{\bm{n}^{\bm{n}}}{(R \bm{k})^{\bm{n}} } \leq 2C \frac{\bm{n}^{\bm{n}}}{1 + (R \bm{k})^{\bm{n}} }$, so the constant $\widetilde{C} =2 \max(C, \lVert \hat{u} \rVert_{L_2(\T^d)}  )$ works.

Let us fix $\bm{n} = \lfloor \alpha \bm{k} \rfloor$ for some constant $\alpha >0$ (which means that $n_j =  \lfloor \alpha k_j\rfloor$ is the biggest integer smaller or equal to $\alpha k_j$ for every $j$). Then, we have
\begin{align*}
	\abs{u[\bm{k}] } & \leq \widetilde{C} \frac{ (\alpha \bm{k})^{\bm{n}} }{1 + (R \bm{k})^{\bm{n}}} \\
				& \leq \widetilde{C} \left( \frac{\alpha}{R} \right)^{\abs{\bm{n}}},
\end{align*}
where the last equality follows from $ (\alpha \bm{k})^{\bm{n}} =  \left( \frac{\alpha}{R} \right)^{\abs{\bm{n}}} (R \bm{k})^{\bm{n}} \leq  \left( \frac{\alpha}{R} \right)^{\abs{\bm{n}}}(1 + (R \bm{k})^{\bm{n}})$.
By choosing $\alpha < R$ and $\epsilon = \alpha \log(\alpha /R) >0$, we deduce that 
\begin{equation*}
	\abs{u[\bm{k}] } \leq \widetilde{C} \left( \frac{\alpha}{R} \right)^{\alpha \abs{ \bm{k}}} \leq \widetilde{C}  \mathrm{e}^{- \epsilon \abs{\bm{k}}}
\end{equation*}
for every $\bm{k} \in \Z^d$, \emph{i.e.} $u[\cdot] \in \Spc E(\Z^d)$. 
\end{proof}

\bibliographystyle{plain}

\bibliography{Wiener}
\end{document}